\documentclass[12pt]{amsart}

\usepackage{amsmath,amssymb}
\usepackage{hyperref}
\usepackage{enumerate}
\usepackage{bm}

  \usepackage{graphicx}
  \usepackage{latexsym} 
  \usepackage[all]{xy}  
  \usepackage{amsfonts} 
  \usepackage[2emode]{psfrag}

\DeclareMathOperator{\newot}{\raisebox{0.2ex}{\scalebox{0.75}{\ensuremath{\otimes}}}}

\newtheorem{proposition}{Proposition}[section]
\newtheorem{lemma}[proposition]{Lemma}
\newtheorem{corollary}[proposition]{Corollary}
\newtheorem{theorem}[proposition]{Theorem}

\theoremstyle{definition}
\newtheorem{definition}[proposition]{Definition}
\newtheorem{example}[proposition]{Example}
\newtheorem{examples}[proposition]{Examples}
\newtheorem{construction}[proposition]{Construction}

\theoremstyle{remark}
\newtheorem{remark}[proposition]{Remark}

\newcommand{\thlabel}[1]{\label{th:#1}}
\newcommand{\thref}[1]{Theorem~\ref{th:#1}}
\newcommand{\selabel}[1]{\label{se:#1}}
\newcommand{\seref}[1]{Section~\ref{se:#1}}
\newcommand{\lelabel}[1]{\label{le:#1}}
\newcommand{\leref}[1]{Lemma~\ref{le:#1}}
\newcommand{\prlabel}[1]{\label{pr:#1}}
\newcommand{\prref}[1]{Proposition~\ref{pr:#1}}
\newcommand{\colabel}[1]{\label{co:#1}}
\newcommand{\coref}[1]{Corollary~\ref{co:#1}}
\newcommand{\relabel}[1]{\label{re:#1}}
\newcommand{\reref}[1]{Remark~\ref{re:#1}}
\newcommand{\exlabel}[1]{\label{ex:#1}}
\newcommand{\exref}[1]{Example~\ref{ex:#1}}
\newcommand{\delabel}[1]{\label{de:#1}}
\newcommand{\deref}[1]{Definition~\ref{de:#1}}
\newcommand{\eqlabel}[1]{\label{eq:#1}}
\newcommand{\equref}[1]{(\ref{eq:#1})}
\newcommand{\conlabel}[1]{\label{con:#1}}
\newcommand{\conref}[1]{Construction~\ref{con:#1}}

\newcommand{\Vect}{{\sf Vect}}
\newcommand{\Dual}{{\sf Dual}}
\newcommand{\id}{{\sf id}\,}

\newcommand{\coev}{{\rm coev}}
\newcommand{\ev}{{\rm ev}}
\newcommand{\blambda}{{\bm \lambda}}
\newcommand{\bgamma}{{\bm \gamma}}
\newcommand{\bLambda}{{\bf \Lambda}}
\newcommand{\bUpsilon}{{\bf \Upsilon}}
\newcommand{\bL}{{\bf L}}
\newcommand{\bC}{\bf C}

\newcommand{\Hom}{{\sf Hom}}

\newcommand{\LHom}{{\sf LHom}}
\newcommand{\End}{{\sf End}}

\newcommand{\Ker}{{\sf Ker}\,}

\newcommand{\im}{{\sf Im}\,}

\newcommand{\Char}{{\sf Char}\,}
\newcommand{\Mod}{{\sf Mod}}
\newcommand{\CoMod}{{\sf CoMod}}
\newcommand{\CMnd}{{\sf CMnd}}
\newcommand{\LieMod}{{\sf LieMod}}
\newcommand{\LieMnd}{{\sf LieMnd}}
\newcommand{\LieCoMod}{{\sf LieCoMod}}
\newcommand{\LieCoMnd}{{\sf LieCoMnd}}

\newcommand{\Id}{{\sf Id}\,}

\newcommand{\Rep}{{\sf Rep}}
\newcommand{\Ind}{\ul{\sf Ind}}
\newcommand{\eq}{{\sf eq}}
\newcommand{\coeq}{{\sf coeq}}
\newcommand{\LRgd}{{\sf LRgd}}
\newcommand{\RRgd}{{\sf RRgd}}
\newcommand{\EML}{{\textrm{EML}}}

\def\Ab{\underline{\underline{\sf Ab}}}
\def\Mich{\underline{\underline{\sf Mich}}}
\def\Tak{\underline{\underline{\sf Tak}}}

\def\ot{\otimes}

\def\lie{\Lambda}
\def\liea{\varrho}
\def\liec{\delta}

\def\dual{\Diamond}

\def\id{\textrm{{\small 1}\normalsize\!\!1}}

\def\SS{{\mathbb S}}

\def\ZZ{{\mathbb Z}}

\newcommand{\Aa}{\mathcal{A}}

\newcommand{\Cc}{\mathcal{C}}
\newcommand{\Dd}{\mathcal{D}}

\newcommand{\Ll}{\mathcal{L}}

\newcommand{\Pp}{\mathcal{P}}

\newcommand{\Uu}{\mathcal{U}}

\def\text#1{{\rm {\rm #1}}}

\def\ul{\underline}
\def\dul#1{\underline{\underline{#1}}}

\def\Set{\dul{\rm Set}}

\def\lim{{\rm lim\,}}
\def\m{{\rm m}}

\def\Bialg{{\sf BiAlg}}

\def\Mnd{{\sf Mnd}}
\def\MND{{\sf MND}}
\def\CMND{{\sf CMND}}
\def\EML{{\sf EML}}
\def\H{{\sf H}}
\def\LH{{\sf LH}}

\def\YBLieAlg{{\sf YBLieAlg}}
\def\YBAlg{{\sf YBAlg}}
\def\YBLieCoAlg{{\sf YBLieCoAlg}}
\def\LieAlg{{\sf LieAlg}}

\def\LieCoAlg{{\sf LieCoAlg}}

\begin{document}
\title[Lie monads]{Lie monads and dualities}
\author{I.\ Goyvaerts} 
\address{Department of Mathematics, Faculty of Engineering, Vrije Universiteit Brussel, Pleinlaan 2, B-1050
Brussel, Belgium}
\email{igoyvaer@vub.ac.be}
\author{J.\ Vercruysse}
\address{D\'epartement de Math\'ematiques, Facult\'e des sciences, Universit\'e Libre de Bruxelles, Boulevard du
Triomphe, B-1050 Bruxelles, Belgium}
\email{jvercruy@ulb.ac.be}



\begin{abstract}
We study dualities between Lie algebras and Lie coalgebras, and their respective (co)representations. To allow a study
of dualities in an infinite-dimensional setting, we introduce the notions of Lie monads and Lie comonads, as special
cases of YB-Lie algebras and YB-Lie coalgebras in additive monoidal categories. We show that (strong) dualities between
Lie algebras and Lie coalgebras are closely related to (iso)morphisms between associated Lie monads and Lie comonads. In
the case of a duality between two Hopf algebras -in the sense of Takeuchi- we recover a duality between a Lie algebra
and a Lie coalgebra -in the sense defined in this note- by computing the primitive and the indecomposables elements, respectively.
\end{abstract}

\maketitle

\section*{Introduction and motivation}

Lie coalgebras were introduced by Michaelis \cite{Mich} as a formal dualization of Lie algebras. In particular, if $(L,\lie)$ is a finite-dimensional Lie algebra over a base field $k$, the dual vector space $C=L^*$ of $L$ can be endowed in a natural way with the structure of a Lie coalgebra, defining the ``Lie co-bracket'' as the linear map $\Upsilon=\lie^*:C=L^*\to C\ot C\cong (L\ot L)^*$, that satisfies an antisymmetry and ``co-Jacobi'' relation. Conversely, any finite-dimensional Lie coalgebra in a canonical way gives rise to a Lie algebra on its dual space. 

As for usual algebras and coalgebras, the passage to infinite-dimensional vector spaces complicates the situation. If $C$ is an infinite-dimensional Lie coalgebra, then the dual space $C^*$ will again be a Lie algebra. On the contrary, for an arbitrary Lie algebra $L$, the dual space is $L^*$ no longer a Lie coalgebra. Rather, one should restrict to the finite dual $L^\circ$, which was shown -again by Michaelis- to be a Lie coalgebra. However, as we know from general considerations, $L^\circ$ is often too small to contain enough information to recover the complete space $L$. Hence, in many situations, another duality theory will be more appropriate. 

The recent revival of monad theory among Hopf algebraists has shown us an alternative approach to attack these kind of dualities \cite{BBW}, \cite{BLV}. Indeed, given a (usual) algebra $A$ over a base field $k$, one can associate to it the monad $-\ot A$ (tensor product over $k$) on the category of $k$-vector spaces. As the endofunctor $-\ot A$ has a right dual $\Hom(A,-)$, this right dual naturally comes equipped with a comonad structure, without any finiteness condition on $A$. In fact, one makes the transition from algebras and coalgebras over the base field $k$ to algebras and coalgebras in the monoidal category of endofunctors (on the category of vector spaces).  In categorical terms, a vector space is finite-dimensional if and only if it has a (right) dual. The analogue property for endofunctors is having a (right) adjoint functor; a right adjoint functor for a functor of the form $-\ot X$ on the category of vector spaces is guaranteed by the Hom-functor $\Hom(X,-)$.

Motivated by the above, our aim is to study a duality for Lie algebras and Lie coalgebras in such a setting. However, if we want to introduce a notion of ``Lie monad'', we encounter a problem: the category of endofunctors is (strict) monoidal (in a canonical way), but not braided nor symmetric. Nevertheless, given a Lie algebra $L$ or Lie coalgebra $C$ in the category of vector spaces, one can define in a very natural way a \textit{Lie monad} structure on the associated endofunctors $-\ot L$ and $\Hom(C,-)$, by means of a local symmetry associated to the twist on the object $L$ and $C$ respectively.
This leads us to the introduction of the notion of a {\em Yang-Baxter-Lie algebra} (YB-Lie algebra for short) in an arbitrary additive monoidal category. 

The notion of a YB-Lie algebra clearly covers the concept of a Lie algebra in a symmetric monoidal category, which in turn unifies several variations of classical Lie algebras, for example Lie superalgebras. It is not our aim to go deeper into this aspect of the theory here. Instead, we refer the interested reader to the recent survey \cite{FCKhar}.

Our paper is organised as follows. After recalling some generalities on monoidal categories, we study YB-Lie algebras in \seref{LieAlgLieMod}. We introduce the category of Lie modules over a YB-Lie algebra and show -in case this YB-Lie algebra is just a Lie algebra in a symmetric monoidal category- that this category is equivalent to the category of representations of the Lie algebra. Furthermore, we study several functors and adjunctions associated to Lie modules.

In \seref{LieCoalgLieComod} we briefly review the dual situation of YB-Lie coalgebras and Lie comodules and provide some examples. \seref{LieMonadsComonads} is devoted to the particular case of Lie monads and Lie comonads. More precisely, we show the bijective correspondence between YB-Lie algebras in an additive monoidal category and Lie monads of the form $-\ot L$ (see \prref{LieAlgvsLiemonad}) as well as the bijective correspondence between Lie modules of a YB-Lie algebra and the (Lie version of the) Eilenberg-Moore category of the associated Lie monad.

In \seref{DualLie} we start our study of dualities. We introduce the notion of a duality between a YB-Lie algebra $L$ and YB-Lie coalgebra $C$ in a closed monoidal category. \prref{dualityliealgebras} shows the close correspondence between dualities for the pair $(L,C)$ and morphisms between associated Lie monads $-\ot L$ and $\Hom(C,-)$, which also induces a functor between the corresponding (co)module categories. Furthermore, {\em strong} dualities are in correspondence with the fact that the associated Lie monad morphism is an isomorphism (\prref{charstrongMich}), and in this situation the (co)module categories are equivalent.

It is well known that the primitive elements of a Hopf algebra form a Lie algebra. Similarly, the indecomposables of a Hopf algebra form a Lie coalgebra. Now, given a braided Hopf algebra, whose Yang-Baxter operator is involutive, we show in \seref{DualLieHopf} that the primitive elements form a YB-Lie algebra in our sense, respectively the indecomposables form a YB-Lie coalgebra. Moreover, given two Hopf algebras that are in duality in the sense of Takeuchi, the associated YB-Lie algebra and YB-Lie coalgebra are in duality in our sense. Finally, we show that these dualities are in correspondence with module and comodule categories (see \thref{TakMichModules}).

\section{Preliminaries}

\subsection*{Monoidal categories, braidings and symmetries} Throughout the paper we will work in a monoidal category $\Cc=(\Cc,\ot,I,a,l,r)$  with associativity constraint $a:\ot\circ(\ot\times\id_{\Cc})\to\ot\circ(\id_{\Cc}\times \ot)$ and with left- and right unit constraints resp. $l$ and $r$ ($\id_{\Cc}$ denotes the identity functor on $\Cc$). Often, if the context allows us, we will suppress the associativity and unit constraints. This will not harm the generality of our considerations, by Mac Lane's coherence theorem. In particular, all our results are applicable in situations where associativity or unit constraints are not trivial, and we will give explicit examples of these situations relevant in our setting below. Often we consider $\Cc$ moreover to be symmetric, and denote the symmetry by $c_{-,-}$.

\subsection*{Additivity}
Throughout, $\Cc$ will be supposed to be an additive category and, in case it exhibits also a monoidal structure, will be such that the tensor product is additive in each variable. In other words, $(f+g)\ot h=f\ot h+ g\ot h$ whenever $f,g,h$ are morphisms of $\Cc$ with $f$ and $g$ parallel.
For any two object $X,Y$ in $\Cc$, we denote the 
Hom-set from $X$ to $Y$ (which is supposed to be an abelian group) as $\Hom_\Cc(X,Y)$ or shortly by $\Hom(X,Y)$ if there
can be no confusion about the category $\Cc$. The identity morphism on $X$ is denoted by $1_X$ or $X$ for short. For any
functor $F:\Cc\to \Dd$, we denote $\Id_{F}$ the natural transformation defined by $\Id_{FX}=1_{FX}$. Although we avoid
this for simplicity, most of the theory developed in this paper, can be easily extended to the setting of ($k$-linear)
enriched categories.

\subsection*{Closedness}
Recall that a monoidal category is called {\em left closed} if any endofunctor of the form $-\ot X$ has a right adjoint. We will denote this right adjoint by $\H(X,-)$. In this situation, for any three objects $X, Y, Z$ in $\Cc$, there is an isomorphism
\begin{equation}\eqlabel{pi}
\pi^X_{Y,Z}:\Hom(Y\ot X,Z)\cong \Hom(Y,\H(X,Z)).
\end{equation}
The unit and counit of the adjunction $(-\ot X,\H(X,-))$ are denoted by
\begin{eqnarray}
\eta^X_Y: Y\to \H(X,Y\ot X);\quad
\epsilon^X_Y: \H(X,Y)\ot X\to Y.
\end{eqnarray}
One can easily observe that for a fixed object $Y$ in $\Cc$, one also obtains a contravariant functor $\H(-,Y):\Cc\to \Cc$ sending $X$ to $\H(X,Y)$. The functoriality comes from the fact that for any morphism $f:X\to X'$, one can construct
$$\H(X,\epsilon^{X'}_Y) \circ \H(X,\H(X',Y)\ot f) \circ \eta^{X}_{\H(X',Y)}: \H(X',Y)\to \H(X,Y)$$
Based on this observation, one easily obtains that $\eta^X_Y$, $\epsilon^X_Y$ and $\pi^X_{Y,Z}$ are also natural in the argument $X$.

Similarly, a monoidal category is called {\em right closed} if any endofunctor $X\ot -$ has a right adjoint, that we will denote in such a situation by $\H'(X,-)$. A monoidal category is called closed if it is both left and right closed. A braided monoidal category is closed if it is left closed or if it is right closed.

The following lemma shows that the adjunction \equref{pi} can in fact be lifted to an enriched adjunction, considering $\Cc$ as a self-enriched category. We refer to \cite[page 14]{Kelly} e.g. for a proof of this result.

\begin{lemma}\lelabel{liftingadjoint}
Let $\Cc$ be a (left) closed monoidal category, and use notation as above. Then there also exist the following natural isomorphisms in $\Cc$: 
$\Pi^X_{Y,Z}:\H(Y\ot X,Z)\cong\H(Y,\H(X,Z))$.
\end{lemma}

Explicitly, one can compute $\Pi$ and $\Pi^{-1}$ in terms of $\eta$ and $\epsilon$, by means of the following formulas (in the strict monoidal setting)
\begin{eqnarray}
\Pi^X_{Y,Z}&=& \H(Y,\H(X,\epsilon^{Y\ot X}_Z))\circ \H(Y,\eta^X_{\H(Y\ot X,Z)\ot Y})\circ \eta^Y_{\H(Y\ot X,Z)} \eqlabel{Piformule}\\
(\Pi^{(-1)})^{X}_{Y,Z}&=&
\H(Y\ot X,\epsilon^X_{Z})\circ \H(Y\ot X,\epsilon^Y_{\H(X,Z)}\ot X)\circ \eta^{Y\ot X}_{\H(Y,\H(X,Z))} 
\eqlabel{Piinvformule}
\end{eqnarray}

\subsection*{Rigidity}
An object $X$ in a monoidal category is called {\em left rigid} if there exists an object ${^*X}$
together with morphisms $\coev:I\to X\ot {^*X}$ and $\ev:{^*X}\ot X\to I$ such that 
$$(X\ot\ev)\circ a^{-1}\circ (\coev\ot X)=X, \quad (\ev\ot {^*X})\circ a\circ ({^*X}\ot\coev)= {^*X}$$
It is easily verified that if $X$ is left rigid, then the object ${^*X}$ is unique up to isomorphism. In this
situation, we call ${^*X}$ the {\em left dual} of $X$ and $({^*X},X,\ev,\coev)$ a {\em duality} (or an {\em adjoint
pair}) in $\Cc$.

A {\em right rigid} object $X$ is defined symmetrically and we denote the right dual of $X$ by $X^*$. Remark that if $X$
is left rigid with left dual ${^*X}$, then ${^*X}$ is right rigid with right dual $({^*X})^*=X$.
A monoidal category is said to be {\em left rigid} (resp. right rigid, resp. rigid) if every object is left (resp.
right, resp. both left and right) rigid. Another name for a rigid monoidal category is an {\em autonomous (monoidal)
category}. If $\Cc$ is braided, then it is right rigid if and only if it is left rigid. If a category is (left, right)
rigid, then it is (left, right) closed and $\H(X,-)\simeq-\ot {^*X}$ (resp $\H'(X,-)\simeq X^*\ot -$).

\subsection*{Generators}

Recall that an object $G\in \Cc$ is called a generator if and only if the functor $\Hom_\Cc(G,-):\Cc\to \Set$ is fully faithful. If the category $\Cc$ has coproducts, this is furthermore equivalent with the fact that for any object $X\in \Cc$ there is a canonical epimorphism $f_X:H=\coprod_{f:G\to X} G \to X$, where the coproduct takes over a number of copies of $G$. Therefore, we find a fork 
\begin{equation}\eqlabel{coequalizergenerator}\xymatrix{\coprod_{^{(g,h):G\to H,\ st}_{f_X\circ g=f_X\circ h}} G \ar@<.5ex>[rr]^-{g_X} \ar@<-.5ex>[rr]_-{h_X} &&\coprod_{f:G\to X} G \ar[rr]^{f_X} &&  X}\end{equation}
In general this diagram is not a coequalizer, but $G$ is called a {\em regular} generator if \equref{coequalizergenerator} is a coequalizer for every $X\in \Cc$, see e.g. \cite[page 81]{Kelly}.

\section{YB-Lie algebras and Lie modules}\selabel{LieAlgLieMod}

\subsection{YB-Lie algebras in additive monoidal categories}

Recall that in a braided monoidal category, the $n$th braid group acts canonically on the $n$th tensor power of any
object. In a symmetric monoidal category, this action induces an action of the $n$th permutation group. As the notion
of a (classical) Lie algebra makes use of cyclic permutations of its elements (in order to fomulate the Jacobi
identity),  the development of a theory of Lie algebras in a braided setting is a lot more
involved than in the symmetric setting and leads to different possible treatments (see e.g.\ \cite{Par} and
\cite{Majid}). In this paper, we omit non-symmetric braidings, rather we allow a symmetry on an object to be a
``local'' gadget.

\begin{definition}
Let $L$ be an object in an additive monoidal category $\Cc$ and $c:L\ot L\to L\ot L$ a morphism satisfying the
following conditions:
\begin{eqnarray}
c\circ c&=&L\ot L; \eqlabel{zelfinvers}\\
&&\hspace{-3cm}
a_{L,L,L}\circ(c\ot L)\circ a_{L,L,L}^{-1}\circ(L\ot c)\circ a_{L,L,L}\circ(c\ot L)\eqlabel{YB} \\
&=&(L\ot c)\circ a_{L,L,L}\circ (c\ot L)\circ a_{L,L,L}^{-1}\circ(L\ot c)\circ a_{L,L,L} \nonumber
\end{eqnarray}
Condition \equref{YB} is exactly the Yang-Baxter equation and \equref{zelfinvers} means that $c$ is involutive. Hence we call a morphism $c$ satisfying the conditions \equref{zelfinvers}-\equref{YB} a {\em symmetric Yang-Baxter operator for} $L$.
\end{definition}

\begin{example}\exlabel{symmetric}
If $\Cc$ is a symmetric monoidal category, with symmetry $c_{-,-}$, then $c_{L,L}$ is a symmetric Yang-Baxter operator for $L$. 
\end{example}

Given an object $L$ in $\Cc$, together with a symmetric Yang-Baxter operator $c$ as above, we can construct the following morphisms in $\Cc$ (compare to \cite[section 5]{CG} for more details in case of \exref{symmetric}):
\begin{eqnarray*}
t=t_c:=
a_{L,L,L}\circ (c\ot L)\circ a^{-1}_{L,L,L}\circ (L\ot c); &&
w=w_c:= (L\ot c)\circ a_{L,L,L}\circ (c\ot L)\circ a^{-1}_{L,L,L}.
\end{eqnarray*}
As in the case of a symmetric monoidal category, the existence of a symmetric Yang-Baxter operator for $L$ induces a
canonical action of the $3$th permuation group on $L\ot (L\ot L)$,
$s:\SS_3\to\End(L\ot (L\ot L))$. In terms of this action, we have that $w=s(1,2,3)$ and $t=s(3,2,1)$, where we
represented the elements of $\SS_3$ as cycles. The following proprties are immediate.


\begin{lemma}\lelabel{b1b2}
With notation as above, the following identities hold,
\begin{enumerate}[(i)]
\item $t\circ w=w\circ t=L$
\item $t\circ t=w$;
\item $t\circ t\circ t=L$.
\end{enumerate}
\end{lemma}

 \begin{remark} 
Consider again the situation of \exref{symmetric}. Then we can take the symmetry $c_{L,L}$ on any object $L$ in the
symmetric monoidal category $\Cc$. We can construct the morphisms $t_{c_{L,L}}=t_L$ and $w_{c_{L,L}}=w_L$ upon which the
lemma above applies. However, for general braided monoidal categories this result is no longer valid, as one can see
from the following counterexample:\\
Let $\Vect^{\ZZ_{2}}(k)$ denote the category whose objects are $\ZZ_{2}$-graded vector spaces over a field $k$ (\Char$(k)\neq 2$), and whose morphisms consist of $k$-linear maps that preserve the grading.
Let $U,V,W$ be objects in $\Vect^{\ZZ_{2}}(k)$. 
Now consider the following associativity constraint $a$ for $\ot_{k}$ (unadorned tensorproducts $\ot$ are to be taken over $k$):
\begin{eqnarray*}
&&\hspace{-2cm}
a_{U,V,W}: (U\ot V)\ot W\to U\ot(V\ot W); \\
&&\hspace{-2cm}
~(x\ot y)\ot z\mapsto (-1)^{|x||y||z|}x\ot(y\ot z),
\end{eqnarray*}
where $|x|$ denotes the degree of a homogeneous element $x$ of an object in $\Vect^{\ZZ_{2}}(k)$.\\
Letting $l$, resp. $r$ be the trivial left, resp. right unit constraints with respect to $k$, we obtain a (non-strict) monoidal category $(\Vect^{\ZZ_{2}}(k),\ot_{k},k,a,l,r)$ which we shall denote by $\Cc$.
Moreover,
$\Cc$ is a braided monoidal category if and only if $k$ contains a primitive fourth root of unity $i$ (see \cite{BCT} for example). A braiding $c$ can then be defined as follows; for any couple of objects $(V,W)$ in $\Cc$,
$$c_{V,W}: V\ot W\to W\ot V; v\ot w\mapsto (i)^{|v||w|}w\ot v.$$
One now checks easily that \equref{zelfinvers}, and hence \leref{b1b2} does not hold for $c=c_{V,V}$ for any object $V$ in $\Cc$, with $c$ defined as above.
\end{remark}

\begin{definition}
Let $\Cc$ be an an additive, monoidal category, but not necessarily symmetric. A {\em YB-Lie algebra in} $\Cc$ is a triple $(L,\lambda,\lie)$, denoted $L$ for short if there is no confusion possible, where $L$ is an object of $\Cc$, $\lambda$ is a symmetric Yang-Baxter operator for $L$ in $\Cc$, and $\lie:\ L\ot L\to L$ is a morphism (which we call -- despite our notation -- a {\em Lie bracket}) in $\Cc$ that satisfies
\begin{eqnarray}\eqlabel{AS}
\lie \circ (1_{L\ot L} + \lambda)&=&0_{L\ot L,L},\\
\eqlabel{Jac}
\lie\circ (1_{L}\ot \lie)\circ (1_{L\ot(L\ot L)}+ t_\lambda+ w_\lambda)&=&0_{L\ot (L\ot L),L}.
\end{eqnarray}
and is such that the following diagram commutes:

\begin{equation}\eqlabel{compatibility}
{\xymatrix{(L\ot L)\ot L \ar[d]_{\lie \ot L} \ar[rr]^-{t_\lambda\circ a_{L,L,L}} && L\ot (L\ot L)   \ar[d]^-{L\ot \lie} \\
L\ot L \ar[rr]_-{\lambda} && L\ot L }
}
\end{equation}
A morphism of YB-Lie algebras $\phi:(L,\lambda,\lie)\to(L',\lambda',\lie')$ is a morphism $\phi:L\to L'$ that respects the Lie-bracket, and the Yang-Baxter operator i.e.\ 
\begin{eqnarray}
\lie'\circ(\phi\ot\phi)&=&\phi\circ\lie; \eqlabel{respectLie} \\ 
\lambda'\circ(\phi\ot\phi)&=&(\phi\ot\phi)\circ \lambda \eqlabel{respectYB}
\end{eqnarray} 
The category of YB-Lie algebras in $\Cc$ and morphisms of YB-Lie algebras between them is denoted by $\YBLieAlg(\Cc)$.

Suppose now that $\Cc$ is an additive, symmetric monoidal category. A Lie algebra in $\Cc$ is a YB-Lie algebra in $\Cc$ of the form $(L,c_{L,L},\lie)$, where $c_{L,L}$ is the symmetry of the category $\Cc$. 

The full subcategory of $\YBLieAlg(\Cc)$ whose objects are Lie algebras in $\Cc$, is denoted by $\LieAlg(\Cc)$. Remark that a morphism between two Lie algebras automatically satisfies condition \equref{respectYB}, by the naturality of the symmetry $c_{-,-}$.
\end{definition}

We call \equref{Jac} the (right) $\lambda$-Jacobi identity for $L$. As for usual Lie algebras, the definition of a YB-Lie algebra is left-right symmetric, i.e.\ any YB-Lie algebra also satisfies a left $\lambda$-Jacobi identity; this result was proven in \cite{GV}.

\begin{examples}\exlabel{LieAlg} 
The notion of a YB-Lie algebra covers many known classes of (generalized) Lie algebras, such as: classical Lie algebras over an arbitrary commutative ring $R$ (working in the symmetric monoidal category $\Mod(R)$), Lie superalgebras (working in the monoidal category of $\ZZ_2$-graded vector spaces, considered with the non-trivial symmetry) and certain classes of Hom-Lie algebras (applying the Hom-construction on an additive symmetric monoidal category, see \cite{CG} for more details about this non-strict example). For more details about the examples above, we refer to \cite{GV}. 
It also covers the theory of Lie monads (working in the non-symmetric monoidal category of additive endofunctors on an additive category), which will be treated in more detail in \seref{liemonad}. As another example, the YB-Lie algebra of primitive elements of a braided bialgebra is constructed in \seref{primitive}.
\\Finally, one observes that if $(L,\lambda,\lie)$ is a Lie algebra, then $(L,\lambda,\lie\circ \lambda)$ is again a YB-Lie algebra, which we call the opposite Lie algebra of $L$. 
\end{examples}

The following class of new examples might be useful in the sequel, it allows for obvious generalizations.

\begin{example}\exlabel{newYBLie}
Let $\Cc$ be a symmetric monoidal category, with symmetry $c_{-,-}$ and
let $(L,\lie)$ be a Lie algebra in $\Cc$, $A$ an object in $\Cc$ and $\mu:A\ot A\to A$ any morphism. Then $(A\ot L,\lambda,\{-,-\})$ is a YB-Lie algebra where
$$\lambda:\xymatrix{A\ot L\ot A\ot L \ar[r]^-{A\ot c_{L,A}\ot L} & A\ot A\ot L\ot L \ar[r]^-{A\ot A\ot c_{L,L}} & A\ot A\ot L\ot L  \ar[r]^-{A\ot c_{A,L}\ot L} & A\ot L\ot A\ot L}$$
and
$$\{-,-\}:\xymatrix{
A\ot L\ot A\ot L \ar[r]^-{A\ot c_{L,A}\ot A} & A\ot A\ot L\ot L \ar[r]^-{\mu\ot \lie} & A\ot L
}$$
E.g.\ taking $\Cc$ to be the category of $k$-vector spaces and letting $L$ be $k$-Lie algebra and $A$ a $k$-algebra, we find that $A\ot L$ is a YB-Lie algebra with
$$\lambda(a\ot x\ot b\ot y)=a\ot y\ot b\ot x;\qquad \{a\ot x, b\ot y\}=ab\ot [x,y],$$
for all $a,b\in A$ and $x,y\in L$.
Remark that this YB-Lie algebra $A\ot L$ is even a YB-Lie algebra in the category of $A$-bimodules.
\end{example}

\subsection{Lie modules}
Let $(\Cc,\ot,I,a,l,r)$ be an additive, monoidal category and 
$(L,\lambda,\lie)$ a YB-Lie algebra in $\Cc$.
\begin{definition}
A {\em right Lie module over $L$} 
is an object $X$ in $\Cc$, together with a morphism $\liea:X\ot L\to X$ satisfying
\begin{equation}
\eqlabel{Liemodule}
\liea\circ\bigg((\liea\ot L)\circ a^{-1}_{X,L,L}-(\liea\ot L)\circ a^{-1}_{X,L,L}\circ 
(X\ot \lambda)-X\ot\lie\bigg)=0_{X\ot(L\ot L),X}
.
\end{equation}

Left Lie modules can be introduced symmetrically. 
\end{definition}

\begin{example}\exlabel{Liemodule}
Let $(L,\lambda,\lie)$ be a YB-Lie algebra in $\Cc$. Then $L$ is a Lie module over itself (with $\liea=\lie$). One easily gets \equref{Liemodule} from the Jacobi identity and antisymmetry.
\end{example}

\begin{definition}
Let $(X,\liea_X)$ and $(Y,\liea_Y)$ be two right Lie modules in $\Cc$. A morphism of Lie modules is a morphism $f:X\to Y$ in $\Cc$ s.t. $f\circ \liea_X=\liea_Y\circ (f\ot L)$.
The set of all morphisms of Lie modules from $X$ to $Y$ is denoted by $\LHom(X,Y)$.
Then, with these definitions, Lie modules in a monoidal category $\Cc$ together with their morphisms form a category, which we will denote by $\LieMod(L)$ (whether we consider left or right modules is supposed to be clear from the context).
\end{definition}

\begin{remark}
If $L$ is a Lie algebra in a symmetric monoidal category, then the category of left Lie modules over $L$ and the category of right Lie modules over $L$ are isomorphic. Over a general YB-Lie algebra however, this seems no longer to be true. Consider for example the YB-Lie algebra from \exref{newYBLie}. If $M$ is an associative right $A$-module and $(N,\liea)$ is a Lie-module, then the tensor product $M\ot N$ has a natural structure of right $A\ot L$-Lie module
, but there is no canonical right $A\ot L$-Lie module structure on $M\ot N$.
\end{remark}

\begin{definition}
Let $\Cc$ be an additive, symmetric closed monoidal category.
A {\em representation of a Lie algebra} $(L,\lie_{L})$ is a pair $(X,\phi_{X})$, where $X$ is an object of $\Cc$ and $\phi_{X}:(L,\lie_L)\to (\H(X,X),\lie_{\H(X,X)})$ is a morphism of Lie algebras, where $\lie_{\H(X,X)}$ is the commutator Lie bracket, defined as follows:
$\lie_{\H(X,X)}=m_{\H(X,X)}-m_{\H(X,X)}\circ c_{\H(X,X),\H(X,X)}$, with $$m_{\H(X,X)}=\pi^X_{\H(X,X)\ot\H(X,X),X}(\epsilon^{X}_{X}\circ(\H(X,X)\ot \epsilon^{X}_{X})).$$
Morphisms are defined as follows: Let $(X,\phi_{X})$ and $(Y,\phi_{Y})$ be two representations of $(L,\lie_L)$ and let $f:X\to Y$ a morphism in $\Cc$. Then $f$ is a morphism of representations if the following diagram commutes
$$
\xymatrix{
    L\ot X \ar[r]^-{\phi_{X}\ot X } \ar[d]_{L\ot f}  & \H(X,X)\ot X \ar[r]^-{\epsilon^{X}_{X}} & X \ar[d]_{f} \\
    L\ot Y \ar[r]_-{\phi_{Y}\ot Y} & \H(Y,Y)\ot Y \ar[r]_-{\epsilon^{Y}_{Y}} & Y }
$$
The category of representations of $L$ is denoted by $\Rep(L)$. 
\end{definition}

\begin{proposition}
Let $L$ be a Lie algebra in a symmetric closed monoidal category.
There is an equivalence of categories between the category of (left) Lie modules $\LieMod(L)$ and the category of representations $\Rep(L)$.
\end{proposition}

\begin{proof}
We define a functor $F:\LieMod(L)\to\Rep(L)$ as follows: 
$$F(X,\liea_{X})=(X,\phi_{X}=\pi^X_{L,X}(\liea_{X})),$$
for any (left) Lie module $(X,\liea_{X})$, and $F$ acts as the identity functor on morphisms. 
By naturality of $\epsilon^{X}_{-}$, we have $\epsilon^{X}_{X}(\pi^{X}_{L,X}(\liea_{X})\ot X)=\liea_{X}$. 
Applying this together with the naturality of $\pi$, one can check that $F$ is well-defined.

Conversely, consider the functor $G:\Rep(L)\to\LieMod(L)$ defined for any object $(X,\phi_{X})$ of $\Rep(L)$ as
$$G(X,\phi_{X})=(X,\liea_{X}=(\pi^{X}_{L,X})^{-1}(\phi_{X})),$$  
and $G$ is the identity on morphisms.
To see that $G$ is well-defined, it suffices to make use of the naturality of $c$, $\epsilon^X_{-}$ and $\pi^{X}_{-,X})^{-1}$.

Finally, it is clear that $(G,F)$ is pair of adjoint functors with trivial unit and counit (i.e. identical natural transformations), hence they establish the desired equivalence of categories.
\end{proof}

\subsection{Adjoint functors for Lie modules}\selabel{adjointmodules}

The following needs no explicit proof. 

\begin{proposition}\prlabel{induction}
Let $(L,\lambda,\lie)$ be a YB-Lie algebra in an additive monoidal category $\Cc$ and $(M,\liea_M)$ a Lie module. Then for any object $X$ in $\Cc$, the pair $(X\ot M,\liea_{X\ot M }=X\ot \liea_{M})$ is a Lie module. This yields a functor 
$-\ot M:\Cc\to \LieMod(L)$.
In particular, taking $(M,\liea_M)=(L,\lie)$, we obtain a functor
$$-\ot L:\Cc\to \LieMod(L).$$
\end{proposition}
%
A natural question that arises is whether these functors have a right adjoint. To obtain this result, we need to shift our setting towards {\em closed} monoidal categories. In the remaining of this section, we wil suppose that $\Cc$ is an additive, left closed monoidal category.

The proof of the following theorem is based on the observation that the set of morphisms between $L$-Lie modules can be expressed as the following equalizer: Let $(M,\liea_M)$ and $(N,\liea_N)$ be two $L$-Lie modules, then we have the following equalizer in $\Ab$
$$\xymatrix{\LHom(M,N)\ar[rr]&& \Hom(M,N) 
\ar@<.5ex>[rrrr]^-{(-)\circ \liea_M }
\ar@<-.5ex>[rrrr]_-{\liea_N\circ ((-)\ot L) } 
&&&& \Hom(M\ot L, N) }$$
To obtain a right adjoint for the functor of \prref{induction}, we need to lift this equalizer to the category $\Cc$.

\begin{theorem}\thlabel{LieHomTensor}
Suppose that $\Cc$ possesses equalizers.
Then the functor $-\ot M:\Cc\to \LieMod(L)$ has a right adjoint $\LH(M,-)$, given by the following equalizer in $\Cc$
\begin{equation}\eqlabel{defLH}
\xymatrix{\LH(M,N)\ar[rr]&& \H(M,N) 
\ar@<.5ex>[rrrr]^-{\pi^{M\ot L}_{\H(M,N),N}\left(\epsilon^M_N\circ (1_{\H(M,N)}\ot \liea_M)\right)}
\ar@<-.5ex>[rrrr]_-{\pi^{M\ot L}_{\H(M,N),N}\left(\liea_N\circ (\epsilon^M_N\ot L)\right) } 
&&&& \H(M\ot L, N) }
\end{equation}
for any Lie module $(N,\liea_N)$.
\end{theorem}

\begin{proof}
We have to prove that there is a natural isomorphism $\LHom(X\ot M,N)\cong \Hom(X,\LH(M,N))$ for any object $X\in \Cc$ and any $L$-Lie module $(N,\liea_N)$. 
\\Consider the following equalizer in $\Ab$:
$$\xymatrix{\LHom(X\ot M,N)\ar[rr]&& \Hom(X\ot M,N) 
\ar@<.5ex>[rrrr]^-{(-)\circ( X\ot \liea_M) }
\ar@<-.5ex>[rrrr]_-{\liea_N\circ ((-)\ot L) } 
&&&& \Hom(X\ot M\ot L, N) }$$
Recall (cf. e.g \cite[Proposition 2.9.4]{Borceux}) that a representable functor preserves all limits. Hence if we apply the representable functor $\Hom(X,-)$ on the equalizer \equref{defLH} defining $\LH(M,N)$, we obtain the following equalizer in $\Ab$:
$$\xymatrix{\Hom(X,\LH(M,N))\ar[rr]&& \Hom(X,\H(M,N)) 
\ar@<.5ex>[rrrr]^-{\left(\pi^{M\ot L}_{\H(M,N),N}\left(\epsilon^M_N\circ (1_{\H(M,N)}\ot \liea_M)\right)\right)^{*}}
\ar@<-.5ex>[rrrr]_-{\left(\pi^{M\ot L}_{\H(M,N),N} \left(\liea_N\circ (\epsilon^M_N\ot L)\right)\right)^{*}} 
&&&& \Hom(X,\H(M\ot L, N)) },$$
where $(-)^{*}$ denotes $\Hom(X,-)$. We know that $\pi^{M}_{X,N}$ and $\pi^{M\ot L}_{X,N}$ respectively provide isomorphisms between the last two objects in the above two equalizers. Or aim is now to show that these isomorphisms  induce an isomorphism between the respective equalizers. Take $f\in \Hom(X\ot M,N)$, then we find
\begin{eqnarray*}
\left(\pi^{M\ot L}_{\H(M,N),N}(\epsilon^{M}_{N}\circ (\H(M,N)\ot\liea_{M}))\right)^*\circ\pi^{M}_{X,N}(f)\\
&&\hspace{-5cm}=\pi^{M\ot L}_{X,N}\left(\epsilon^{M}_{N}\circ (\H(M,N)\ot\liea_{M})\circ(\pi^{M}_{X,N}(f)\ot M\ot L)\right)\\
&&\hspace{-5cm}=\pi^{M\ot L}_{X,N}(\epsilon^{M}_{N}\circ (\pi^{M}_{X,N}(f)\ot M)\circ(X\ot \liea_{M}))\\
&&\hspace{-5cm}=\pi^{M\ot L}_{X,N}(f\circ(X\ot\liea_{M}))
\end{eqnarray*}
where we used the naturality of $\pi^{M\ot L }_{-,N}$ in the first equality and the naturality of the tensor product in the second equality and the naturality of $\epsilon^{M}_{-}$ in combination with the fact that $\epsilon^{M}_{X\ot M}\circ({\eta^{M}_{X}}\ot M)=1_{X\ot M}$ and $\pi^M_{X,N}=\H(M,f)\circ\eta^M_Y$ in the last equality.

A similar computation shows that
$$\pi^{M\ot L}_{X,N}(\liea_{N}\circ(f\ot L))=\pi^{M\ot L}_{\H(M,N),N}(\liea_{N}\circ(\epsilon^{M}_{N}\ot L))\circ \pi^{M}_{X,N}(f).$$
Now, by  the uniqueness of the equalizer, we obtain a natural isomorphism $\LHom(X\ot M,N)\cong \Hom(X,\LH(M,N))$, which shows the adjunction between $-\ot M$ and $\LH(M,-)$.
\end{proof}

\begin{construction}[{{\bf The commutator Lie algebra}}]\conlabel{commutator}
Let $(B,\mu_B)$ be a (non-unital) associative algebra in $\Cc$. We say that $B$ is a YB-algebra if it comes equipped with an involutive Yang-Baxter operator $\lambda_B:B\ot B\to B\ot B$ that satisfies the following condition
\begin{equation}\eqlabel{assalgebra}
{\xymatrix{
    B\ot B\ot B \ar[r]^-{t_{\lambda_B}} \ar[d]_{\mu_{B}\ot B}  & B\ot B\ot B \ar[r]^-{w_{\lambda_B}} \ar[d]_{B\ot \mu_{B}} & B\ot B\ot B \ar[d]_{\mu\ot B} \\
    B\ot B \ar[r]_-{\lambda_B} & B\ot B \ar[r]_-{\lambda_B} & B\ot B }}
\end{equation}
The category of YB-algebras in $\Cc$ is denoted by $\YBAlg(\Cc)$. One can easily check that for any YB-algebra $(B,\mu_B,\lambda_B)$, the triple $(B,\lie_B,\lambda_B)$, where $\lie_{B}=\mu_{B}\circ(B\ot B-\lambda_B)$ is a YB-Lie algebra. We call this YB-Lie algebra the commutator Lie algebra of $B$, and denote it for short as $\Ll(B)$. This construction defines a functor $\Ll:\YBAlg(\Cc)\to \YBLieAlg(\Cc)$.
\end{construction}

\begin{examples}
\begin{enumerate}
\item Let $A$ be an algebra in a symmetric monoidal category with symmetry $c$, then $A$ is a YB-algebra, its YB-operator being $c_{A,A}$. The associated YB-Lie algebra is the usual commutator Lie algebra.
\item Let $A$ and $B$ be two algebras in a symmetric monoidal category $\Cc$, then $A\ot B$ is again an algebra in $\Cc$. We define a YB-operator $\lambda$ on $A\ot B$ by
$$\lambda:\xymatrix{A\ot B\ot A\ot B \ar[r]^-{A\ot c_{B,A}\ot B} & A\ot A\ot B\ot B \ar[r]^-{A\ot A\ot c_{B,B}} & A\ot A\ot B\ot B  \ar[r]^-{A\ot c_{A,B}\ot B} & A\ot B\ot A\ot B}$$
One easily verifies that $A\ot B$ is a YB-algebra. The associated YB-Lie algebra $\Ll(A\ot B) = A\ot \Ll(B)$ -of the form of the YB-Lie algebra constructed in \exref{newYBLie}- is made out of the tensor product of the algebra $A$ and the usual commutator Lie algebra $\Ll(B)$ of $B$.
\item Let $A$ be an associative algebra. It was shown in \cite{Nich} that the following map defines a symmetric YB-operator on $A$
$$\lambda(a\ot b)= ab\ot 1 - a\ot b + 1\ot a,$$
for all $a,b\in A$.
In fact, endowed with this YB-operator the associative algebra $A$ becomes a YB-algebra. One easily checks that the associated YB-Lie algebra has a trivial bracket, that is $[a,b]=0$ for all $a,b\in A$, i.e. $\Ll(A)$ is a ``commutative'' YB-Lie algebra.
\end{enumerate}
\end{examples}

Let us fix a YB-Algebra $B$ and 
denote the category of (right) $B$-modules $(M,\rho_{M})$, $(\rho_{M}$ being the right action of $B$ on $M$) by $\Mod(B)$. 
Then we can define a functor 
$$\ul{\sf Ind}(-):\xymatrix{\Mod(B)\ar[rr] &&\LieMod(\Ll(B))}$$
by putting $\ul{\sf Ind}(M,\rho_{M})=(M,\liea_{M}=\rho_{M})$. The fact that $\ul{\sf Ind}(-)$ is well-defined follows from the (mixed) associativity of the (right) action of $B$ onto any (right) $B$-module.
Remark that because of this, a YB-algebra $B$ always possesses two $\Ll(B)$-Lie module structures: one by its commutator Lie-bracket, and one by its initial (associative) multiplication.

We will search for an adjoint for the functor $\ul{\sf Ind}$. However, we will work in a more general setting. Let $(L,\lambda,\lie)$ be any YB-Lie algebra and $B$ an associative algebra. Let $(T,\liea_T)$ be a $L$-Lie module that is at the same time a left $B$-module with action $\m:B\ot T\to T$ such that $\m$ is a morphism of $L$-Lie modules, where the $L$-Lie module structure on $B\ot T$ is given by $B\ot \liea_T$, i.e.\ it is the structure induced by the functor $-\ot T$ of \prref{induction} evaluated in $B$. This means that 
\begin{equation}\eqlabel{liebimodule}
\liea_T\circ (\m\ot L)=\m\circ (B\ot \liea_T).
\end{equation}
Then for any right $B$-module $(M,\rho_M)$, we find that $(M\ot_B T, M\ot_B\liea_T)$ is an $L$-Lie module. Hence we find a well-defined functor 
$$-\ot_BT:\Mod(B)\to \LieMod(L).$$
In case we take $L=\Ll(B)$, and $(T,\liea_T)=(B,\mu_B)$ with the regular left $B$-action, then we find that this functor is exactly $\ul{\sf Ind}$.
Before stating the next theorem, we need a little lemma:
\begin{lemma}\lelabel{lemma1}
Suppose $\Cc$ to be complete and let $B$ be an associative algebra in $\Cc$. Then the forgetful functor $\Uu:\Mod(B)\to \Cc$ reflects limits.
\end{lemma}
\begin{proof}
As $\Cc$ is complete, we know that $\Mod(B)$ is complete as well (see e.g.\ \cite[Fact 2]{Porst}). Now, we observe that $\Uu$ preserves limits, since it has a left adjoint $-\ot B$. Finally, we also have that $\Uu$ reflects isomorphisms. The lemma now follows from e.g.\ \cite[Proposition 2.9.7]{Borceux}.
\end{proof}

\begin{theorem}\thlabel{LieHomTensor2}
Suppose that $\Cc$ is an additive, left closed monoidal category with equalizers. 
Let $(L,\lambda,\lie)$ be a YB-Lie algebra in $\Cc$, $B$ an associative algebra in $\Cc$, $(T,\liea_T)$ a $L$-Lie module that is a left $B$-module with action $\m$ such that \equref{liebimodule} holds. Then then the functor $\LH(T,-)$ from \thref{LieHomTensor} can be corestricted to obtain a right adjoint to the above defined functor
$$-\ot_BT:\Mod(B)\to\LieMod(L).$$
\end{theorem}

\begin{proof}
Let $(M,\liea_{M})$ be an object in $\LieMod(L)$. Define $\LH(T,(M,\liea_{M}))=(\LH(T,M),\rho_{M})$, with
$$\rho_{M}=\LH(T,\zeta_{M} \circ (\LH(T,M)\ot \m))\circ \theta_{\LH(T,M)\ot B},$$
where $\theta_X:X\to \LH(T,X\ot T)$ is the unit and $\zeta_M:\LH(T,M)\ot T\to M$ the counit of the adjunction between $-\ot T$ and $\LH(T,-)$ from \thref{LieHomTensor}.
Then it follows smoothly, from naturality and the fact that $\m$ is a left $B$-action on $T$, that $\rho_M$ defines an associative and unital right $B$-action.

To prove the adjunction, we need to prove that we have an isomorphism of abelian groups 
$$\LHom(M\ot_{B}T,N)\cong \Hom_{B}(M,\LH(T,N)),$$
whenever $(N,\liea_{N})\in \LieMod(L)$ and $(M,\rho_{M})\in \Mod(B)$.
To this end, we use a similar argument as in \thref{LieHomTensor}. 
First, remark that Lie module homomorphisms from $M\ot_B T$ to $N$ can be characterized as the following equalizer in $\Ab$:
$$\xymatrix{\LHom(M\ot_{B}T,N)\ar[rr]&& \Hom(M\ot_{B}T,N) 
\ar@<.5ex>[rr]^-{(-)\circ(M\ot_{B}\liea_{T})}
\ar@<-.5ex>[rr]_-{\liea_{N}\circ((-)\ot L)} 
&& \Hom(M\ot_{B}T\ot L, N) }.$$
Next, we consider the following equalizer in $\Cc$:
\begin{equation}\eqlabel{LHequalizerinproof}
\xymatrix{\LH(T,N)\ar[rr]&& \H(T,N) 
\ar@<.5ex>[rrrr]^-{{\pi}^{T\ot L}_{\H(T,N),N}({\epsilon}^T_N\circ (id_{\H(T,N)}\ot_{B} \liea_T))}
\ar@<-.5ex>[rrrr]_-{{\pi}^{T\ot L}_{\H(T,N),N}(\liea_N\circ ({\epsilon}^T_N\ot L)) } 
&&&& \H(T\ot L, N) },
\end{equation}
where ${\pi}$ and ${\epsilon}$ denote as before the (natural) isomorphisms associated to the adjunction between $-\ot T$ and $\H(T,-)$. We know from the first part of the proof that $\LH(T,N)$ is moreover a right $B$-module. In a similar way, classical arguments of enriched category theory tell us that $\H(T,N)$ and $\H(T\ot L,N)$ are right $B$-modules (we even have an adjunction $(-\ot_B T,\H(T,-))$ between $\Mod(B)$ and $\Cc$). Hence, \equref{LHequalizerinproof} is an equalizer in $\Mod(B)$ by \leref{lemma1}.
We can thus apply the representable functor $\Hom_{B}(M,-)$ to this equalizer, to obtain the following equalizer in $\Ab$.
$$\xymatrix{\Hom_{B}(M,\LH(T,N))\ar[r]& \Hom_{B}(M,\H(T,N)) 
\ar@<.5ex>[rrr]^-{\left({\pi}^{T\ot L}_{\H(T,N),N}({\epsilon}^T_N\circ (id_{\H(T,N)}\ot_{B} \liea_T))\right)^{*}}
\ar@<-.5ex>[rrr]_-{\left({\pi}^{T\ot L}_{\H(T,N),N}(\liea_N\circ ({\epsilon}^T_N\ot L))\right)^{*} } 
&&& \Hom_{B}(M,\H(T\ot L, N)) }$$
To conclude the proof, it now suffices to observe that by the adjunction $(-\ot_BT,\H(T,-))$, we have isomorphisms $\Hom(M\ot_{B} T,N)\cong \Hom_{B}(M,\H(T,N))$ and $\Hom(M\ot_{B} T\ot L,N)\cong \Hom_{B}(M,\H(T\ot L,N))$, which implies that the above constructed equalizers in $\Ab$ are isomorphic.
\end{proof}

As a particular instance of \thref{LieHomTensor2}, we find that the functor $\ul{\sf Ind}:\Mod(B)\to \LieMod(B)$, being naturally isomorphic to $-\ot_BB$, has a right adjoint $\LH(B,-)$. 
Moreover, we obtain the following diagram of adjoint functors for any right $B$-module $(T,\rho_T)$.
\[
\xymatrix{
&& \Cc \ar@<-.5ex>[lldd]_-{-\ot T} \ar@<.5ex>[rrdd]^-{-\ot \widehat T} \\ \\
\Mod(B)\ar@<-.5ex>[rrrr]_-{\ul{\sf Ind}\simeq -\ot_B B} \ar@<-.5ex>[rruu]_-{\H_B(T,-)} &&&&\LieMod(\Ll(B)) \ar@<.5ex>[lluu]^-{\LH(\widehat T,-)} \ar@<-.5ex>[llll]_-{\LH(B,-)}
}
\]
Here we denote $\widehat{T}=\ul{\sf Ind}(T)$, the induced $L$-Lie module of $T$ and the functor $\H_B(T,-):\Mod(B)\to \Cc$ is the internal representable functor defined by the equalizer for all $(X,\rho_X)\in \Mod(B)$
$$\xymatrix{\H_B(T,X) \ar[rr] && \H(T,X) 
\ar@<.5ex>[rrrr]^-{\pi^{T\ot B}_{\H(T,X),X}\left(\epsilon^T_X\circ (1_{\H(T,X)}\ot \rho_T)\right)}
\ar@<-.5ex>[rrrr]_-{\pi^{T\ot B}_{\H(T,X),X}\left(\rho_X\circ (\epsilon^T_X\ot B)\right) }
&&&& \H(T\ot B,X) },$$
which is known to be a right adjoint for $-\ot T:\Cc\to\Mod(B)$.
Clearly $\Ind(X\ot T)\cong X\ot \widehat T$, so the outer triangle in the above diagram naturally commutes. To see that the inner diagram of functors also commutes, take any object $X$ in $\LieMod(\Ll(B))$ and $Y$ in $\Cc$. By applying the adjunctions above, we then find
\begin{eqnarray*}
\Hom(Y,\H_B(T,\LH(B,X)))&\cong& \Hom_B(Y\ot T,\LH(B,X))
\cong \LHom(\Ind(Y\ot B),X)\\
&\cong &\LHom(Y\ot \widehat T,X) \cong \Hom(Y,\LH(\widehat T,X))
\end{eqnarray*}
So by the Yoneda lemma, we find a natural isomorphism $\H_B(T,\LH(B,X))\cong \LH(\widehat T,X)$.

\section{YB-Lie coalgebras and Lie comodules}\selabel{LieCoalgLieComod}

Let $(\Cc, \ot,I)$ be an additive, monoidal category. {\em YB-Lie coalgebras in $\Cc$} are defined dually to YB-Lie algebras, i.e.\ we define the category $\YBLieCoAlg(\Cc)$ of YB-Lie coalgebras in $\Cc$ as $\YBLieCoAlg(\Cc)=\YBLieAlg(\Cc^{op})^{op}$. If $\Cc$ is moreover symmetric, we define Lie coalgebras in $\Cc$ by $\LieCoAlg(\Cc)=\LieAlg(\Cc^{op})^{op}$. Here $\Cc^{op}=(\Cc^{op},\ot^{op},I)$ denotes the opposite category of $\Cc$ and $\ot^{op}: \Cc^{op}\times\Cc^{op}\to \Cc^{op}$ the opposite tensor product functor induced in the obvious way by $\ot$. Explicitly, this leads to the following definition, which is due to Michaelis in the symmetric case (cf. \cite{Mich}).
\begin{definition}\delabel{LieCoalgebra}
A YB-Lie coalgebra in $\Cc$ is a triple $(C,\gamma,\Upsilon)$, denoted $C$ for short if no confusion can be made, consisting of an object $C$ in $\Cc$ together with a self-invertible YB-operator $\gamma:C\ot C\to C\ot C$ and a comultiplication map $\Upsilon:C\to C\ot C$ such that 
\begin{enumerate}
\item $\Upsilon+\gamma\circ \Upsilon=0_{C,C\ot C}$;
\item $(1_{C\ot (C\ot C)}+ w_\gamma + t_\gamma )\circ (1_{C}\ot \Upsilon)\circ \Upsilon=0_{C, C\ot (C\ot C)}$;
\item $(\Upsilon\ot C)\circ \gamma = a^{-1}_{C,C,C}\circ w_\gamma \circ (C\ot \Upsilon)$.
\end{enumerate}
In case that $\Cc$ is moreover symmetric, then we call a YB-Lie coalgebra of the form $(C,c_{C,C},\Upsilon)$, where $c_{C,C}$ is the symmetry of $\Cc$, just a Lie-coalgebra.

A {\em Lie comodule over $C$} is an object $X$ in $\Cc$, endowed with a morphism $\liec^X:X\to X\ot C$ satisfying
$$\bigg(a_{X,C,C}\circ(\liec^X\ot C)-(X\ot\gamma)\circ a_{X,C,C}\circ(\liec^X\ot C)-X\ot \Upsilon\bigg)\circ \liec^X=0_{X,X\ot(C\ot C)}.
$$
\end{definition}
Morphisms of Lie comodules are definied in the obvious way. The category of Lie comodules over $C$ with their morphisms will be denoted by $\LieCoMod(C)$.

All statements and theorems of the previous section have obvious duals for Lie coalgebras and Lie comodules. There is no point in repeating these explicitly. Let us just finish this section by mentioning some examples (see also \cite{Mich} for \exref{Liecoalgebra} (1),(2) and (4)) of Lie coalgebras that will be useful later on.
\begin{examples}\exlabel{Liecoalgebra}
\begin{enumerate}
\item Let $(C,\Delta_{C})$ be a coassociative coalgebra in an additive, monoidal category $\Cc$ and suppose there is a involutive Yang-Baxter operator $\gamma:C\ot C\to C\ot C$ on $C$, such that an analogeous version of the diagram \equref{assalgebra} commutes, then we can consider a YB-Lie coalgebra structure on $\Ll^c(C)=C$, defined by $\Upsilon_{\Ll^c(C)}=(C\ot C-\gamma)\circ \Delta_{C}$.
\item\label{indecomposables}
Let $H$ be a Hopf algebra in $\Vect(k)$. Let $I=\Ker(\epsilon)$, with $\epsilon$ the counit of $H$, and let us denote $Q(H)=I/I^{2}$, the so-called {\em indecomposables} of $H$. Then $Q(H)$ is a Lie coalgebra, where the cobracket comes from $\Delta_{\Ll^c(H)}$. To see that this is true, let us first check that $\Delta_{\Ll^c(H)} : I\to I\ot I $ is well-defined. Indeed, since 
$$\im(\Delta_{\Ll^c(I)})\subset \Ker(\epsilon\ot1_{H})\cap \Ker(1_{H}\ot \epsilon)=(I\ot H)\cap (H\ot I)=I\ot I,$$
where we denoted $\Delta_{\Ll^c(I)}$ the restriction of $\Delta_{\Ll^c(H)}$ to $I$.
Moreover, $\Delta_{\Ll^c(I)}$ maps $I^{2}$ into $I^{2}\ot I+I\ot I^{2}$, so the map 
$$\overline{\Delta_{\Ll^c(I)}}: I/I^{2}\to I\ot I/({I^{2}\ot I+I\ot I^{2}})\cong {I/I^{2}}\ot {I/I^{2}}$$
is well-defined and turns $Q(H)$ into a Lie coalgebra.  Let us point out that dually to the Lie algebra case, $Q(H)$ can be described as the following coequalizer
\begin{equation}\eqlabel{indeccoeq}\xymatrix{
H\ot H\ar@<.5ex>[rr]^-{\mu} \ar@<-.5ex>[rr]_-{\epsilon\ot H+H\ot \epsilon} && H\ar[rr]^{\coeq} && Q(H)
}\end{equation}
Let us remark that the construction of indecomposables in terms of a coequalizer as above, allows to perform this construction in any category with sufficiently well-behaving coequalizers. We will come back to this in \seref{primitive}.
\item The next example is closely related to the previous one.
Let $H$ be again a Hopf algebra. Consider the space $X\subset H$ consisting of all $x$ such that
\begin{equation}\eqlabel{coprimitive}
(f*g)(x)=f(x)+g(x),
\end{equation}
where $f,g\in H^*$ and $*$ is the convolution product. Then one can compute that the comultiplication restricted to $X$ is cocommutative (that is $\tau\circ\Delta=\Delta$ on all elements of $X$, where $\tau:H\ot H\to H\ot H$ is the switch map). Indeed: consider a base $\{e_i\}$ for $X$, then $x=\sum a_ie_i$, $\Delta(x)=\sum a_{ij} e_i\ot e_j$. Now apply condition \equref{coprimitive} for the dual base elements $f=e^*_i$ and $g=e^*_j$, then one finds that $a_{ij}=a_i+a_j$. The cocommutativity now follows.
We consider the quotient space $C=H/X$. 
Since $X$ is cocommutative, the map $\Upsilon=\Delta-\tau\circ \Delta$ is well-defined on $C$. One can now check that $(C,\Upsilon)$ is a Lie-coalgebra. Let us call $C$ the {\em Lie coalgebra of coprimitives}. 
Moreover, there is a Lie coalgebra morphism 
$$C\to Q(H);\qquad \overline{h}\mapsto \overline{h-\eta(\epsilon(h))}$$ 
\item\label{dual} Let $L$ be a finite-dimensional Lie $k$-algebra. Then its dual space $C=L^*$ can be endowed with the structure of a Lie coalgebra, by putting $\Upsilon:L^*\to L^*\ot L^*\cong (L\ot L)^*$; the dual map of the Lie bracket. Similarly, if $L$ is a YB-Lie algebra, then $C$ is a YB-Lie coalgebra with $\gamma=\lambda^*$. Conversely, if $C$ is any Lie coalgebra (or YB-Lie coalgebra), even infinite-dimensional, then its dual space $C^*$ becomes a Lie algebra. We will treat this in more detail in \seref{DualLie}.
\item Let $\Aa$ be any additive category, and  $\End(\Aa)$ an additive, monoidal category of additive endofunctors on $\Aa$ and natural transformations between them. We will call a YB-Lie coalgebra in $\End(\Aa)$ a {\em Lie comonad} on $\Aa$, see \reref{comonads}.
\end{enumerate}
\end{examples}

\section{Lie monads and comonads}\selabel{LieMonadsComonads}
\subsection{Lie monads}\selabel{liemonad}

We already introduced Lie monads in \exref{LieAlg} as YB-Lie algebras in a category of additive endofunctors on an additive category, let us restate the definition in explicit form. We will provide two generic classes of examples for Lie monads, one arising from YB-Lie algebras and one from YB-Lie coalgebras.

\begin{definition}
A {\em YB-Lie monad}, or {\em Lie monad} for short, on an additive category $\Cc$ is a triple $(\bL,\blambda,\bLambda)$, where $\bL:\Cc\to\Cc$ is an additive functor, $\blambda:\bL\circ \bL\to \bL\circ \bL$ is a natural transformation satisfying the involutive Yang-Banxter equations
\begin{eqnarray}
\blambda\circ\blambda&=&\Id_{\bL\circ \bL}\\
(\Id_{\bL}*\blambda)\circ (\blambda*\Id_{\bL})\circ (\Id_{\bL}*\blambda)&=&(\blambda*\Id_{\bL})\circ(\Id_{\bL}*\blambda)\circ(\blambda*\Id_{\bL})
\end{eqnarray}
and $\bLambda:\bL\circ \bL\to \bL$ is a natural transformation satisfying the following conditions:
\begin{eqnarray}
\bLambda\circ(\Id_{\bL\circ \bL}+\blambda)&=&0_{\bL\circ \bL,\bL}\\
\bLambda\circ (\bLambda*\Id_{\bL})\circ (\Id_{\bL\circ \bL\circ \bL}+t_{\blambda}+w_{\blambda})&=&0_{\bL\circ \bL\circ \bL,\bL}
\end{eqnarray}
where $t_{\blambda}=(\Id_{\bL}*\blambda)\circ(\blambda*\Id_{\bL})$, $w_{\blambda}=(\Id_{\bL}*\blambda)\circ(\blambda*\Id_{\bL})$, and is such that the following diagram commutes:
\begin{equation}\eqlabel{Liemonadcompatibility}
{\xymatrix{\bL\circ \bL\circ \bL \ar[d]_{\Id_{\bL}* \bLambda} \ar[rr]^-{\blambda*\Id_{\bL}} && \bL\circ \bL\circ \bL \ar[rr]^-{\Id_{\bL}* \blambda}
 && \bL\circ \bL\circ \bL \ar[d]^-{\bLambda * \Id_{\bL}} \\
\bL\circ \bL \ar[rrrr]_-{\blambda} &&&& \bL\circ \bL }
}
\end{equation}
$\bm\zeta:(\bL,\blambda,\bLambda)\to (\bL',\blambda',\bLambda')$ is a morphism of Lie monads on $\Cc$ if $\zeta:\bL\to \bL'$ is a natural transformation satisfying the two following two conditions:
\begin{itemize}
\item $\bLambda'_{X}\circ (\zeta*\zeta)_X=\zeta_{X}\circ\bLambda_{X}$ 
\item $\blambda'_{X}\circ{(\zeta*\zeta)_X}={(\zeta*\zeta)_X}\circ \blambda_{X},$
\end{itemize}
whenever $X$ is an object of $\Cc$. 

Lie monads and their morphisms form a category, which will be denoted $\LieMnd(\Cc)$.
\end{definition} 

\begin{example}\exlabel{liemonadfromYBliealgebra}
Let $\Cc$ be an additive monoidal category and $(L,\lambda_L,\lie_{L})$ be a YB-Lie algebra in $\Cc$. Then we have that $\ul\Mnd((L,\lambda_L,\lie_{L})):=(-\ot L, \blambda, \bLambda)$ is a Lie monad on $\Cc$, where $\blambda$ and $\bLambda$ 
are defined on any object $M$ in $\Cc$ as follows:
\begin{eqnarray*}
\blambda_M:= M\ot\lambda_L:&& M\ot L\ot L\to M\ot L\ot L\\
\bLambda_{M}:= M\ot \lie_{L}:&& M\ot L\ot L\to M\ot L
\end{eqnarray*}
This is easily checked by using the antisymmetry and Jacobi-identity of $\lie_{L}$ in $\Cc$. 
The condition $\equref{Liemonadcompatibility}$ is also satisfied; it is condition \equref{compatibility}, combined with the naturality of the associativity constraint $a$. Recall from \cite[Example 3.10]{GV} that the underlying reason for this example to work is that the functor ${\sf End}:\Cc\to \End(\Cc)$ sending an object $X$ in $\Cc$ to the endofunctor $-\ot X$ is a strong monoidal functor. 
\end{example}

As a slight variation of the previous example, we have the following.

\begin{example}
Consider the YB-Lie algebra $A\ot L$ from \exref{newYBLie}, which is in fact a YB-Lie algebra in the category of $A$-bimodules. Then $-\ot_A(A\ot L)\simeq -\ot L$ defines a Lie-monad on the category of (say, right) $A$-modules.
\end{example}

\begin{proposition}\prlabel{liemonadfromYBliealgebra}
Let $\Cc$ be an additive monoidal category. Then the assignment from \exref{liemonadfromYBliealgebra} defines a functor
$$\ul\Mnd:\YBLieAlg(\Cc)\to\LieMnd(\Cc).$$
\end{proposition}
\begin{proof}
Whenever $f:(L,\lambda_L,\lie_{L})\to(L',\lambda_{L'},\lie_{L'})$ is a morphism in $\YBLieAlg(\Cc)$, $\ul\Mnd(f)$ is the natural transformation from $-\ot L$ to $-\ot L'$, defined for any object $X$ in $\Cc$ as $\ul\Mnd(f)_{X}=X\ot f$. It is easily verified that this defines a morphism in $\LieMnd(\Cc)$, using subsequently the facts that $f$ preserves the Lie-bracket and the Yang-Baxter operator.
\end{proof}

The following provides a partial converse of \prref{liemonadfromYBliealgebra}.

\begin{proposition}\prlabel{LieAlgvsLiemonad}
Let $\Cc$ be an additive monoidal category. Suppose that the unit object $I$ is a regular generator and that the endofunctors $-\ot X$ and $X\ot -$ preserve colimits in $\Cc$ for any object $X$ in $\Cc$. Let $L$ be an object in $\Cc$.

If $(-\ot L,\blambda, \bLambda)$ is a Lie monad on $\Cc$ then $(L,\blambda_I, \bLambda_{I})$ is a YB-Lie algebra in $\Cc$.

Moreover, there is a bijective correspondence between YB-Lie algebra-structures on $L$ and Lie monad-structures on the endofunctor $-\ot L$.
\end{proposition}
\begin{proof}
Using the fact that $I$ is a regular generator, one proves that 
$\blambda_I\ot X\simeq \blambda_{I\ot X}$ and $\bLambda_{I}\ot X\simeq \bLambda_{I\ot X}$ for all objects $X\in \Cc$. 
Applying this fact, one easily verifies the antisymmetry and Jacobi identity for the Lie monad $-\ot L$ from the corresponding properties for the YB-Lie algebra $L$.

For the last statement, one needs to verify that the construction of  \prref{liemonadfromYBliealgebra} together with the construction above leads to the bijective correspondence. This is a typical computation, see e.g. \cite[Theorem 1.11]{Ver:variations} for a similar case.
\end{proof}

\begin{example}\exlabel{liemonadfromYBliecoalgebra}
Let $\Cc$ be an additive, left closed monoidal category and $(C,\gamma_C,\Upsilon_C)$ be a YB-Lie coalgebra in $\Cc$, then $\ul\Mnd'((C,\gamma_C,\Upsilon_C)):=(\H(C,-),\bgamma,\bUpsilon)$, with $\bgamma$ and $\bUpsilon$ defined on any object $M$ by the following diagrams (here the maps $\Pi^C_{C,M}$ are the isomorphisms from \leref{liftingadjoint}) 
\begin{eqnarray*}
\xymatrix{\H(C,\H(C,M))\ar[d]^-{(\Pi^C_{C,M})^{-1}} \ar[rr]^-{\bgamma_M} && \H(C,\H(C,M)) \\
\H(C\ot C,M)  \ar[rr]^-{\H(\gamma_C,M)} && \H(C\ot C,M) \ar[u]^-{\Pi^C_{C,M}}
}\quad
\xymatrix{
\H(C,\H(C,M)) \ar[rr]^-{\bUpsilon_M} \ar[dr]_-{(\Pi^C_{C,M})^{-1}} && \H(C,M) \\
&\H(C\ot C,M)  \ar[ur]_-{\H(\Upsilon_{C},M)}
}
\end{eqnarray*}
is a Lie monad. Indeed, it is easily checked that $\H(C,-)$ is additive and the antisymmetry and Jabobi identity for the Lie monad are verified using the corresponding properties of $\Upsilon_C$ and $\gamma_C$. Similar to  \prref{liemonadfromYBliealgebra}, on shows that this construction is functorial. 
\end{example}

\begin{proposition}\prlabel{liemonadfromYBliecoalgebra}
Let $\Cc$ be an additive closed monoidal category. Then the construction from \exref{liemonadfromYBliecoalgebra} defines a functor
$$\ul\Mnd':\YBLieCoAlg(\Cc)^{op}\to\LieMnd(\Cc).$$
\end{proposition}

\begin{remark}\relabel{comonads}
Dually to all Definitions and Theorems above, one can introduce and study Lie comonads on additive categories. All Lie comonads form a category $\LieCoMnd(\Cc)$. Without mentioning all details explicitly, let us just mention (as this will be used in the sequel) some notation.
\\Let $(C,\gamma_{C},\Upsilon_{C})$ be a YB-Lie coalgebra in $\Cc$. One has a Lie comonad $(C\ot -, \tilde{\bgamma},\tilde{\bUpsilon})$, defined in the obvious way. Letting $(L,\lambda_L,\lie_{L})$ be a YB-Lie algebra in $\Cc$ and provided $\Cc$ is right closed, one also has a Lie comonad $(\H'(L,-),\tilde{\blambda},\tilde{\bLambda})$.
This then induces functors  $\ul\CMnd:\YBLieCoAlg(\Cc)\to \LieCoMnd(\Cc)$ and $\ul\CMnd':\YBLieAlg(\Cc)\to \LieCoMnd(\Cc)$.
\end{remark}

\subsection{The Eilenberg-Moore category of a Lie monad}
Let $(\bL,\blambda,\bLambda)$ be a Lie monad on an additive category $\Cc$. 
We construct the category of {\em Eilenberg-Moore-Lie} objects $\EML(\bL)$ whose objects are couples $(X,\liea_{X})$, where $X$ is an object and $\liea_{X}:{\bL}X\to X$ is a morphism in $\Cc$ such that
$$(\liea_{X}\circ{\bL}\liea_{X})\circ(1_{{\bL}{\bL}X}-\blambda_{X})-\liea_{X}\circ \bLambda_{X}=0_{{\bL}{\bL}X,X}.$$
The morphisms of $\EML(\bL)$ are morphisms $f:X\to Y$ in $\Cc$ such that $f\circ \liea_X=\liea_Y\circ{\bL }f$.

The constructions of Lie monads out of YB-Lie algebras 
as in the previous section correspond nicely with the notion of the Eilenberg-Moore category.

\begin{proposition}\prlabel{EMLvsmodules}
Let $(L,\lambda_L,\lie_{L})$ be a YB-Lie algebra in $\Cc$. Then there is an equivalence of categories
$$\MND_L:\LieMod(L)\to \EML(\Mnd(L)).$$
\end{proposition}
\begin{remark}
Let $(\bC, \bgamma, \bUpsilon)$ be a Lie comonad. Then one can introduce in the canonical way the category of Eilenberg-Moore-Lie objects for this Lie-comonad. Furthermore, if $C$ is a Lie coalgebra, then dually to \prref{EMLvsmodules} we find an equivalence of categories $\CMND_C:\LieCoMod(C)\to \EML(\CMnd(C))$. If $L$ is a Lie algebra, then there is also a canonical functor $\CMND'_L:\LieMod(L)\to \EML(\CMnd'(L))$.
\end{remark}

\section{Dualities between Lie algebras and Lie coalgebras}\selabel{DualLie}

\subsection{Michaelis pairs}

Troughout this section, let $\Cc$ be an additive, closed (strict) monoidal category. 
\begin{definition}\delabel{duality}
\begin{enumerate}[(1)]
\item
A {\em Michaelis pair} $(L,C,\ev)$ consists of a YB-Lie algebra 
$(L,\lambda_L,\lie_L)$, a YB-Lie coalgebra $(C,\gamma_C,\Upsilon_C)$ 
and a morphism
$$\ev: L\ot C\to I$$
in $\Cc$ that renders commutative the following diagrams 
\begin{equation}\eqlabel{Mich1}
\xymatrix{
L\ot L\ot C \ar[rr]^-{L\ot L\ot \Upsilon_C} \ar[d]_{\lie_L\ot C} && L\ot L\ot C\ot C \ar[rr]^-{L\ot \ev\ot C} 
 && L\ot C \ar[d]^\ev\\
L\ot C \ar[rrrr]_-{\ev} &&&& I
}
\end{equation}
and 
\begin{equation}\eqlabel{Mich2}
\xymatrix{
L\ot L\ot C\ot C \ar[rr]^-{L\ot L\ot \gamma_C} \ar[d]_{\lambda_{L}\ot C\ot C} && L\ot L\ot C\ot C \ar[rr]^-{L\ot \ev\ot C} 
 && L\ot C \ar[d]^-{\ev}\\
L\ot L\ot C\ot C\ar[rr]^-{L\ot \ev\ot C} &&L\ot C\ar[rr]^-{\ev}&& I
}
\end{equation}
The morphism $\ev$ is called a {\em duality} between $L$ and $C$; the set of all dualities between $L$ and $C$ is denoted by $\Dual(L,C)$.
\item A morphism between two Michaelis pairs is a couple $(\phi,\psi):(L,C,\ev)\to (L',C',\ev')$, where $\phi:L\to L'$ is a YB-Lie algebra morphism and $\psi:C'\to C$ is a YB-Lie coalgebra morphism 
satisfying
\begin{equation}
\ev'\circ (\phi\ot C')=\ev\circ(L\ot \psi)
\end{equation}
\item Michaelis pairs and their morphisms form a new category $\Mich(\Cc)$.
\end{enumerate}
\end{definition} 

Let $(L,C,\ev)$ be a Michaelis pair. Then using the adjunction properties $\Hom(X\ot L\ot C,X)\simeq \Hom(X\ot L,\H(C,X))$ and $\Hom(L\ot C\ot X,X)\simeq\Hom(C\ot X,\H'(L,X))$, we can associate to the Michaelis pair two morphisms that are natural in $X$, as follows:
\begin{eqnarray}
\zeta_X=\H(C,X\ot \ev)\circ \eta^C_{X\ot L}
:X\ot L\to \H(C,X) \eqlabel{zeta}\\
\theta_X= 
\H'(L,\ev\ot X)\circ \eta^{'L}_{C\ot X}: C\ot X\to \H'(L,X)
\end{eqnarray}
where $\eta$ (respectively $\eta'$) denotes -as before- the counit of the adjunction associated the left (respectively right) closedness of $\Cc$. Using notations of the previous section and denoting all Lie monad morphisms $\Hom(\ul\Mnd(L),\ul\Mnd'(C))$ and Lie comonad morphisms $\Hom(\ul\CMnd(C),\ul\CMnd'(L))$, we now have the following result:
\begin{proposition}\prlabel{dualityliealgebras}
Let $(L,\lambda,\lie_L)$ be a YB-Lie algebra and $(C,\gamma,\Upsilon)$ be a YB-Lie coalgebra in $\Cc$.
\begin{enumerate}[(i)]
\item There are maps
$$\xymatrix{\Dual(L,C) \ar@<.5ex>[rr]^-\alpha && \Hom(\ul\Mnd(L),\ul\Mnd'(C)) \ar@<.5ex>[ll]^-\beta}$$
such that $\beta\circ \alpha=1_{\Dual(L,C)}$. Moreover, if $I$ is a regular generator, then $\alpha\circ\beta=1_{ \Hom(\ul\Mnd(L),\ul\Mnd'(C))}$.
\item There are maps
$$\xymatrix{\Dual(L,C) \ar@<.5ex>[rr]^-{\alpha'} && \Hom(\ul\CMnd(C),\ul\CMnd'(L)) \ar@<.5ex>[ll]^-{\beta'}}$$
such that $\beta'\circ \alpha'=1_{\Dual(L,C)}$. Moreover, if $I$ is a regular generator, then $\alpha'\circ\beta'=1_{\Hom(\ul\CMnd(C),\ul\CMnd'(L))}$.
\end{enumerate}

\end{proposition}
\begin{proof}
$\ul{(i)\alpha}$.
Let $\ev$ be a duality between $L$ and $C$.
We define $\alpha(\ev)=\zeta$ as in \equref{zeta} and use notation as in \exref{liemonadfromYBliealgebra} and \exref{liemonadfromYBliecoalgebra} for $\ul\Mnd(L)$ and $\ul\Mnd'(C)$. Then $\zeta$ will be a morphism of Lie monads if and only if for any object $X\in \Cc$,
$$\bUpsilon_{X}\circ (\zeta*\zeta)_X=\zeta_{X}\circ\bLambda_{X}, \qquad
\bgamma_{X}\circ{(\zeta*\zeta)_X}={(\zeta*\zeta)_X}\circ \blambda_{X}.$$
We only prove the first identity, the second one follows by a similar computation.
We can compute
\begin{eqnarray*}
&&\bUpsilon_{X}\circ (\zeta*\zeta)_X\\
&=& \H(\Upsilon_C,X)\circ (\Pi_{C,X}^C)^{-1}\circ \H(C,\H(C,X\ot \ev))\circ \H(C,\eta^C_{X\ot L}) \circ \H(C,X\ot L\ot \ev)\circ \eta^C_{X\ot L\ot L}\\
&=& \H(\Upsilon_C,X)\circ \H(C\ot C,X\ot \ev)\circ(\Pi_{C,X\ot L\ot C}^C)^{-1}\circ  \H(C,\eta^C_{X\ot L}) \circ \H(C,X\ot L\ot \ev)\circ \eta^C_{X\ot L\ot L}\\
&=& \H(C,X\ot \ev)\circ \underbrace{\H(\Upsilon_C,X\ot L\ot C)\circ (\Pi_{C,X\ot L\ot C}^C)^{-1}\circ  \H(C,\eta^C_{X\ot L})} \circ \H(C,X\ot L\ot \ev)\circ \eta^C_{X\ot L\ot L}
\end{eqnarray*}
Let us first compute the underbraced part separately, then we find, using \equref{Piinvformule}
\begin{eqnarray*}
&&\hspace{-1cm}\H(\Upsilon_C,X\ot L\ot C)\circ (\Pi_{C,X\ot L\ot C}^C)^{-1}\circ  \H(C,\eta^C_{X\ot L})\\
&&\hspace{-1.5cm}=\H(\Upsilon_C,X\ot L\ot C)\circ \H(C\ot C,\epsilon^C_{X\ot L\ot C} \circ (\epsilon^C_{\H(C,X\ot L\ot C)}\ot C)) \circ \eta^{C\ot C}_{\H(C,\H(C,X\ot L\ot C))} \circ  \H(C,\eta^C_{X\ot L})\\
&&\hspace{-1.5cm}=\H(\Upsilon_C,X\ot L\ot C)\circ \H(C\ot C,\epsilon^C_{X\ot L\ot C} \circ (\epsilon^C_{\H(C,X\ot L\ot C)}\ot C) \circ (\H(C\ot C,\eta^C_{X\ot L})\ot C\ot C))\circ \eta^{C\ot C}_{\H(C,X\ot L)}\\
&&\hspace{-1.5cm}=\H(\Upsilon_C,X\ot L\ot C)\circ \H(C\ot C,\epsilon^C_{X\ot L\ot C} \circ (\eta_{X\ot L}^C\ot C) \circ (\epsilon^C_{X\ot L}\ot C))\circ \eta^{C\ot C}_{\H(C,X\ot L)}\\
&&\hspace{-1.5cm}=\H(\Upsilon_C,X\ot L\ot C)\circ \H(C\ot C, \epsilon^C_{X\ot L}\ot C)\circ \eta^{C\ot C}_{\H(C,X\ot L)}\\
&&\hspace{-1.5cm}=\H(C,\epsilon^C_{X\ot L}\ot C)\circ \H(\Upsilon_C,\H(C,X\ot L)\ot C\ot C)\circ \eta^{C\ot C}_{\H(C,X\ot L)}\\
&&\hspace{-1.5cm}=\H(C,\epsilon^C_{X\ot L}\ot C)\circ \H(C,\H(C,X\ot L)\ot \Upsilon_C)\circ \eta^C_{\H(C,X\ot L)}
\end{eqnarray*}
All equalities follow by naturality and adjunction property of closedness, in particular the last equality follows by the naturality of $\eta$ in the upper argument.
We can now continue 
\begin{eqnarray*}
&&\bUpsilon_{X}\circ (\zeta*\zeta)_X\\
&&\hspace{-1.5cm}=\H(C,X\ot \ev)\circ \H(C,\epsilon^C_{X\ot L}\ot C)\circ \H(C,\H(C,X\ot L)\ot \Upsilon_C)\circ \eta^C_{\H(C,X\ot L)} \circ \H(C,X\ot L\ot \ev)\circ \eta^C_{X\ot L\ot L}\\
&&\hspace{-1.5cm}=\H(C,(X\newot \ev)\circ (\epsilon^C_{X\ot L}\newot C)\circ (\H(C,X\newot L)\newot \Upsilon_C)\circ 
(\H(C,X\newot L\newot \ev)\newot C))\circ \eta^C_{\H(C,X\ot L\ot L\ot C)} \circ \eta^C_{X\ot L\ot L}\\
&&\hspace{-1.5cm}=\H(C,(X\newot \ev)\circ (X\newot L\newot \ev\newot C)\circ (\epsilon^C_{X\ot L\ot L\ot C}\newot C)\circ (\H(C,X\newot L\newot L\newot C)\newot \Upsilon_C)\circ 
(\eta^C_{X\ot L\ot L}\newot C))\circ \eta^C_{X\ot L\ot L}\\
&&\hspace{-1.5cm}=\H(C,(X\newot \ev)\circ (X\newot L\newot \ev\newot C)\circ (\epsilon^C_{X\ot L\ot L\ot C}\newot C)\circ (\eta^C_{X\ot L\ot L}\newot C\newot C)\circ (X\newot L\newot L\newot \Upsilon_C))\circ \eta^C_{X\ot L\ot L}\\
&&\hspace{-1.5cm}=\H(C,(X\newot \ev)\circ (X\newot L\newot \ev\newot C)\circ (X\newot L\newot L\newot \Upsilon_C))\circ \eta^C_{X\ot L\ot L}\\
&&\hspace{-1.5cm}=\H(C,(X\ot \ev)\circ (X\ot \lie_L\ot C))\circ \eta^C_{X\ot L\ot L}\\
&&\hspace{-1.5cm}=\H(C,X\ot \ev)\circ \eta^C_{X\ot L}\circ (X\ot\lie_L)=\zeta_X\circ \bLambda_X
\end{eqnarray*}
$\ul{(i)\beta}$.
Suppose that $\zeta:\Mnd(L)\to \Mnd'(C)$ is a Lie monad morphism. The adjunction property $\Hom(X\ot L\ot C,X)\simeq \Hom(X\ot L,\H(C,X))$ now allows us to define 
$\ev=\epsilon_I^C\circ (\zeta_I\ot C)$. 
Then, by the computations of the first part of the proof, we find from $\bUpsilon_X\circ (\zeta*\zeta)_X=\zeta_X\circ \bLambda_X$ that 
$$\H(C,(X\newot \ev)\circ (X\newot L\newot \ev\newot C)\circ (X\newot L\newot L\newot \Upsilon_C))\circ \eta^C_{X\ot L\ot L}
=\H(C,(X\ot \ev)\circ (X\ot \lie_L\ot C))\circ \eta^C_{X\ot L\ot L}$$
and therefore
\begin{equation}\eqlabel{tebewijzen}(X\newot \ev)\circ (X\newot L\newot \ev\newot C)\circ (X\newot L\newot L\newot \Upsilon_C)=(X\ot \ev)\circ (X\ot \lie_L\ot C)
\end{equation}
To see this, put $l$ the left-hand side of \equref{tebewijzen} and $r$ the right-hand side, then tensor $l$ and $r$ on the right-hand side with the identity morphism on $C$, and compose both sides with $\epsilon^C_X$. We then obtain
\begin{eqnarray*}\epsilon^C_X\circ (\H(C,l)\ot C)\circ (\eta^C_{X\ot L\ot L}\ot C) = \epsilon^C_X\circ (\H(C,r)\ot C)\circ (\eta^C_{X\ot L\ot L}\ot C),
\end{eqnarray*}
which is equivalent to 
\begin{eqnarray*}
l \circ \epsilon^C_{X\ot L\ot L\ot C}\circ \eta^C_{X\ot L\ot L}\ot C = r \circ \epsilon^C_{X\ot L\ot L\ot C}\circ \eta^C_{X\ot L\ot L}\ot C,
 \end{eqnarray*}
which implies \equref{tebewijzen}. If we then take $X=I$ in \equref{tebewijzen}, we obtain \equref{Mich1}. Similarly, $\bgamma_{X}\circ{(\zeta*\zeta)_X}={(\zeta*\zeta)_X}\circ \blambda_{X}$ implies \equref{Mich2}.
\item $\ul{(i)}$. We still have to check that both constructions above are mutual inverses. So let $\ev$ be the evaluation map of a given Michaelis pair $(L,C,\ev)$ and denote $\ev'=\beta\circ\alpha(\ev)$, then we find
\begin{eqnarray*}
\ev'&=&\epsilon_I^C\circ (\H(C,\ev)\ot C)\circ (\eta^C_{L}\ot C)\\
&=&\ev\circ \epsilon_{L\ot C}^C\circ (\eta^C_{L}\ot C) = \ev
\end{eqnarray*}
The other way around, suppose that $I$ is a regular generator. Given a Lie monad morphism $\zeta$, we denote $\zeta'=\alpha\circ\beta(\zeta)$, such that
\begin{eqnarray*}
\zeta'_X&=&\H\big(C,X\ot (\epsilon_I^C\circ (\zeta_I\ot C))\big)\circ \eta^C_{X\ot L}\\
&=& \H(C,X\ot\epsilon_I^C)\circ \eta^C_{X\ot H(C,I)} \circ (X\ot \zeta_I) 
\end{eqnarray*}
Hence, clearly $\zeta'_I=\zeta_I$. On the other hand,
$$\zeta_X=\H(C,\epsilon_X^C)\circ \eta^C_{\H(C,X)}\circ \zeta_X=\H(C,\epsilon^C_X)\circ \H(C,\zeta_X\ot L)\circ \eta_{X\ot L}^C,$$
therefore, $\zeta_X=\zeta'_X$ if and only if $\epsilon_X^C\circ (\zeta_X\ot C)=(X\ot \epsilon_I)\circ (X\ot\zeta_I\ot C)$. Denote these natural transformations by $\sigma_X$ and $\tau_X$ respectively. As $I$ is a regular generator, we can construct for any object $X$ a coequalizer $(X,q)$ starting from a suitable fork $I^{(J)}\rightrightarrows I^{(K)}$. This way we obtain a diagram
$$\xymatrix{
I^{(J)} \ar@<.5ex>[rr] \ar@<-.5ex>[rr] && I^{(K)} \ar[rr]^q && X \\
I^{(J)} \ot L\ot C \ar@<.5ex>[u]^{\tau_I^{(J)}} \ar@<-.5ex>[u]_{\sigma_I^{(J)}} \ar@<.5ex>[rr] \ar@<-.5ex>[rr] && I^{(K)} \ot L\ot C \ar[rr]^{q\ot L\ot C} \ar@<.5ex>[u]^{\tau_I^{(K)}} \ar@<-.5ex>[u]_{\sigma_I^{(K)}} && X \ot L\ot C \ar@<.5ex>[u]^{\tau_X} \ar@<-.5ex>[u]_{\sigma_X}
}$$
In this diagram both lines are coequalizers (the lower line because $\Cc$ is a closed category, hence functor of the form $-\ot Y$ have a right adjoint and therefore preserve colimits). By the naturality of $\sigma$ and $\tau$, the diagram commutes serially (i.e. it commutes if we only consider the arrows with $\tau$ and it commutes if we only consider arrows with $\sigma$) and since $\sigma_I=\tau_I$ we then find by the universal property of the coequalizer that $\tau_X=\sigma_X$.\\ 
$\ul{(ii)}$. Is proven in the same way. 
\end{proof}

\begin{proposition}\prlabel{ModuleFunctorMichPair}
Let $(L,C,\ev)$ be a Michaelis pair and $(X,\liec^{X})$ be a left $C$-Lie comodule. Then $(X,\liea_X)$ is a left $L$-Lie module, where 
$$\liea_{X}:~~\xymatrix{L\ot X\ar[rr]^-{L\ot \liec^{X}} && L\ot C\ot X\ar[rr]^-{\ev\ot X}&& I\ot X\cong X}.$$
This construction yields a functor 
$$F:\LieCoMod(C)\to \LieMod(L),$$
from the category of left $C$-Lie comodules to the category of left $L$-Lie modules.
\end{proposition}

\begin{proof}
With notation as in the statement of the proposition, let us check that $(X,\liea_X)$ is indeed a Lie module. We compute
\begin{eqnarray*}
\liea_X\circ (L\ot \liea_X)&=&(\ev\ot X)\circ(L\ot \liec^X)\circ(L\ot \ev\ot X)\circ(L\ot L\ot \liec^X)\\
&=& (\ev\ot X)\circ (L\ot \ev\ot C\ot X)\circ (L\ot L\ot C\ot \liec^X)\circ (L\ot L\ot \liec^X)
\end{eqnarray*}
and
\begin{eqnarray*}
\liea_X\circ (L\ot \liea_X)\circ(\lambda\ot X)&=&(\ev\ot X)\circ(L\ot \liec^X)\circ(L\ot \ev\ot X)\circ(L\ot L\ot \liec^X)\circ(\lambda\ot X)\\
&&\hspace{-3cm}= (\ev\ot X)\circ (L\ot \ev\ot C\ot X)\circ(\lambda\ot C\ot C\ot X)\circ (L\ot L\ot C\ot \liec^X)\circ (L\ot L\ot \liec^X)\\
&&\hspace{-3cm}= (\ev\ot X)\circ (L\ot \ev\ot C\ot X)\circ(L\ot L\ot \gamma\ot X)\circ (L\ot L\ot C\ot \liec^X)\circ (L\ot L\ot \liec^X)
\end{eqnarray*}
Similarly, we find
\begin{eqnarray*}
\liea_X\circ (\lie\ot X)&=&(\ev\ot X)\circ(L\ot \liec^X)\circ(\lie\ot X)\\
&=&(\ev\ot X)\circ(\lie\ot C\ot X)\circ (L\ot L\ot\liec^X)\\
&=&(\ev\ot X)\circ(L\ot \ev\ot C\ot X)\circ (L\ot L\ot\Upsilon\ot X)\circ(L\ot L\ot\liec^X)\\
\end{eqnarray*}
Combining these equalities, we find
\begin{eqnarray*}
&&\liea_X\circ (L\ot \liea_X)-\liea_X\circ (L\ot \liea_X)\circ(\lambda\ot X)-\liea_X\circ (\lie\ot X)\\
&=&(\ev\ot X)\circ (L\ot \ev\ot C\ot X)\circ(L\ot L\ot \gamma\ot X)\\
&&\hspace{1cm}\circ \bigg(L\ot L\ot \big((C\ot \liec^X)-((\gamma\ot X)\circ (C\ot \liec^X))-(\Upsilon\ot X) \big) \bigg)\circ(L\ot L\ot \liec^X)=0
\end{eqnarray*}
where we used the Jacobi identity of the $C$-Lie comodule $X$ in the last equality. Hence we can define $F(X,\liec^X)=(X,\liea_X)$. 
Furthermore, one easily checks that $F$ is well-defined on morphisms.
\end{proof}

\subsection{Strong Michaelis pairs}

Let $\Cc$ be a monoidal category. Let us denote by $\LRgd(\Cc)$ the complete subcategory of $\Cc$ that consists of all left rigid objects in $\Cc$. Similarly, we denote by $\RRgd(\Cc)$ the complete subcategory of $\Cc$ consisting of all right rigid objects. For any two dualities $(Y,X,\ev,\coev)$ and $(Y',X',\ev',\coev')$, we obtain a new duality
$(Y\ot Y',X'\ot X,\ev\circ(Y\ot\ev'\ot X),(X'\ot \coev\ot Y')\circ\coev')$. Consequently, the categories $\LRgd(\Cc)$ and $\RRgd(\Cc)$ are monoidal, and allow monoidal forgetful functors $\LRgd(\Cc)\to \Cc^{op}$ and $\RRgd(\Cc)\to \Cc$. Furthermore, taking the left (resp. right) dual of a left (resp. right) rigid object, induces a pair of inverse equivalences between the categories
\begin{equation}\eqlabel{rigidequiv}
{^*(-)}:\xymatrix{\LRgd(\Cc)^{op} \ar@<.5ex>[rr] && \RRgd(\Cc) \ar@<.5ex>[ll]}:(-)^*
\end{equation}
As a consequence, we obtain the following result.

\begin{proposition}\prlabel{elementaryMich}
Let $(L,C,\ev,\coev)$ be a duality in $\Cc$. Then $L$ is a YB-Lie algebra if and only if $C$ is a YB-Lie coalgebra, and in this case $(L,C,\ev)$ is a Michaelis pair.
\end{proposition}

\begin{proof}
The equivalence \equref{rigidequiv} induces an equivalence between the categories of YB-Lie algebras in the respective categories. Since a YB-Lie algebra in $\Cc^{op}$ is exactly a Lie coalgebra in $\Cc$, we obtain in fact an equivalence between the categories $\YBLieAlg(\LRgd(\Cc)^{op})\simeq \YBLieCoAlg(\LRgd(\Cc))^{op}$ and $\YBLieAlg(\RRgd(\Cc))$. 

Explicitly, if $(L,C,\ev,\coev)$ be a duality in $\Cc$ and $(L,\lambda,\lie)$ is a YB-Lie algebra, then the YB-Lie coalgebra structure on $C$ is given by $(C,\gamma,\Upsilon)$ where
$$\gamma=C\ot C\ot (\ev\circ (L\ot \ev\ot C))\circ (C\ot C\ot\lambda\ot C\ot C)\circ ((C\ot\coev\ot L)\circ \coev)\ot C\ot C$$
and
$$\Upsilon=(C\ot C\ot \ev)\circ (C\ot C\ot\lie\ot C)\circ \big(((C\ot\coev\ot L)\circ \coev)\ot C\big).$$
Let us just check that $(L,C,\ev)$ is indeed a Michaelis pair. Putting
$\ev^2=\ev \circ (L\ot \ev\ot C)$ and $\coev^2=(C\ot \coev\ot L)\circ \coev$, we find
\begin{eqnarray*}
\ev^2\circ(L\ot L\ot\Upsilon)
&=&\ev^2 \circ\bigg(
L\ot L\ot\big((C\ot C\ot \ev)\circ (C\ot C\ot\lie\ot C)\circ (\coev^2\ot C)\big)\bigg)\\
&&\hspace{-2cm}=\ev\circ (\lie\ot C)\circ (\ev^2\ot L\ot L\ot  C)\circ (L\ot L\ot \coev^2\ot C)
=\ev\circ (\lie\ot C)
\end{eqnarray*} 
Hence, $(L,C,\ev)$ satisfies \equref{Mich1} and by a similar computation one verifies \equref{Mich2}.
\end{proof}

\begin{definition}
A Michaelis pair is called {\em strong}, if it is isomorphic to a Michaelis pair that emerges as in \prref{elementaryMich}.
\end{definition}

The next proposition characterizes strong Michaelis pairs. 

\begin{proposition}\prlabel{charstrongMich}
There is a bijective correspondence between:
\begin{enumerate}[(i)]
\item Strong Michaelis pairs $(L,C,\ev)$;
\item YB-Lie algebras $L$ such that $L$ is a right rigid object in $\Cc$;
\item YB-Lie coalgebras $C$ such that $C$ is a left rigid object in $\Cc$;
\item Michaelis pairs $(L,C,\ev)$ such that the associated Lie monad morphism $\zeta=\alpha(\ev)$ (see \prref{dualityliealgebras}) is an isomorphism;
\item Michaelis pairs $(L,C,\ev)$ such that the associated Lie comonad morphism $\theta=\alpha'(\ev)$ (see \prref{dualityliealgebras}) is an isomorphism.
\end{enumerate}
\end{proposition}

\begin{proof}
The equivalence between the first 3 items follows directly from \prref{elementaryMich}.\\
$\ul{(i)\Rightarrow(iv)}$. Let us prove that $\xi_X=(\epsilon_X^C\ot L)\circ (\H(C,X)\ot \coev)$ is an inverse for $\zeta_X$.
\begin{eqnarray*}
\xi_X\circ \zeta_X&=&(\epsilon_X^C\ot L)\circ (\H(C,X)\ot \coev)\circ \H(C,X\ot \ev)\circ \eta_{X\ot L}^C\\
&=&(\epsilon_X^C\ot L)\circ \H(C,X\ot \ev)\ot C\ot L\circ (\H(C,X\ot L\ot C)\ot \coev)\circ \eta_{X\ot L}^C\\
&=&(X\ot \ev\ot L)\circ (\epsilon_{X\ot L\ot C}\ot L) \circ (\eta^C_{X\ot L}\ot C\ot L)\circ (X\ot L\ot \coev)\\
&=&(X\ot \ev\ot L)\circ (X\ot L\ot \coev)= X\ot L \\
\zeta_X\circ \xi_X&=& \H(C,X\ot \ev)\circ \eta_{X\ot L}^C\circ (\epsilon_X^C\ot L)\circ (\H(C,X)\ot \coev)\\
&=&\H(C,X\ot \ev)\circ \H(C,\epsilon^C_X\ot L\ot C)\circ \eta_{\H(C,X)\ot C\ot L}^C\circ (\H(C,X)\ot \coev)\\
&=&\H(C,\epsilon^C_X)\circ \H(C,\H(C,X)\ot C\ot \ev)\circ \H(C,\H(C,X)\ot \coev\ot C)\circ \eta^C_{\H(X,X)}\\
&=&\H(C,\epsilon^C_X)\circ \eta^C_{\H(X,X)}=\H(C,X)
\end{eqnarray*}
where we used the expression for $\zeta_X$ from \equref{zeta}.\\
$\ul{(iv)\Rightarrow(i)}$. We define $\coev=\zeta_C^{-1}\circ\eta_I^C$. Then we find
\begin{eqnarray*}
(C\ot\ev)\circ(\coev\ot C)&=&(C\ot\ev)\circ(\coev\ot C)\circ \eta_C^C\circ (\eta_I^C\ot C)\\
&&\hspace{-3.5cm}=\eta_C^C\circ (\H(C,C\ot \ev)\ot C)\circ (\H(C,\coev\ot C)\ot C)\circ (\eta_I^C\ot C)\\
&&\hspace{-3.5cm}=\eta_C^C\circ (\H(C,C\ot \ev)\ot C)\circ (\H(C,\zeta_C^{-1}\ot C)\ot C)\circ (\H(C,\eta_I^C\ot C)\ot C)\circ (\eta_I^C\ot C)\\
&&\hspace{-3.5cm}=\eta_C^C\circ (\H(C,C\ot \ev)\ot C)\circ (\H(C,\zeta_C^{-1}\ot C)\ot C)\circ (\eta_{\H(C,C)}\ot C)\circ (\eta_I^C\ot C)\\
&&\hspace{-3.5cm}=\eta_C^C\circ (\H(C,C\ot \ev)\ot C)\circ (\eta_{C\ot L}^C\ot C)\circ (\zeta_C^{-1}\ot C) \circ (\eta_I^C\ot C)\\
&&\hspace{-3.5cm}=\eta_C^C\circ (\zeta_C\ot C)\circ (\zeta_C^{-1}\ot C) \circ (\eta_I^C\ot C)=C
\end{eqnarray*}
A similar computation shows that $(\ev\ot L)\circ (L\ot \coev)=L$.\\
$\ul{(i)\Leftrightarrow(v)}$ is similar to $\ul{(i)\Leftrightarrow(iv)}$.
\end{proof}

\begin{proposition}\prlabel{ModuleFunctorStrongMich}
Let $(L,C,\ev)$ be a strong Michaelis pair. Then the functor 
$$F:\LieCoMod(C)\to \LieMod(L)$$
from \prref{ModuleFunctorMichPair} is an equivalence of categories.
\end{proposition}

\begin{proof}
We define a functor $G:\LieMod(L)\to\LieCoMod(C)$ as follows. Take any left $L$-lie module $(X,\liea_X)$. Then we define a $C$-Lie coaction $\liec^X$ on $C$ by
$$\liec^X:\xymatrix{X\ar[rr]^-{\coev\ot X} && C\ot L\ot X\ar[rr]^-{C\ot \liea_X} && C\ot X}.$$
One proves similarly as in \prref{ModuleFunctorMichPair} that $G$ is well-defined.\\
Next, we observe that $FG(X,\liea_X)\cong (X,\liea_X)$. Indeed, if we denote $FG(X,\liea_X)= (X,\liea'_X)$ then
\begin{eqnarray*}\liea'_X&=&(\ev\ot X)\circ(L\ot C\ot\liea_X)\circ(L\ot \coev\ot X) \\
&=&\liea_X\circ(\ev\ot L\ot X)\circ(L\ot \coev\ot X)=\liea_X
\end{eqnarray*}
Similarly, $GF(X,\liec^X)\cong (X,\liec^X)$ and $(F,G)$ is an equivalence of categories.
\end{proof}

From \prref{elementaryMich} and \prref{ModuleFunctorStrongMich}, we now immediately have the following result, which is the ``Lie version'' of the classical analogous result for usual monads (see e.g.\ \cite{EM}).

\begin{corollary}\colabel{strongMichmonad}
Let $(L,R)$ be an adjoint pair of additive endofunctors on an additive category $\Aa$.
Then $L$ is a Lie monad if and only if $R$ is a Lie comonad, and in this situation the Eilenberg-Moore categories are equivalent.
\end{corollary}

\begin{remark}
It is an interesting question to ask whether the above study of strong dualities between Lie algebras and Lie coalgebras can be generalized to a more general setting, introducing ``rationality'' for Lie coalgebras and considering non-degenerate evaluation morphisms.
\end{remark}

\begin{example}[Finite-dimensional Lie algebras]
If $L$ is a finite-dimensional $k$-Lie algebra, then $L$ is a (left and right) rigid object in the symmetric monoidal category of $k$-vector spaces. Hence $C=L^*$, the vector space dual of $L$ is a Lie coalgebra, as we already remarked in \exref{Liecoalgebra}(\ref{dual}), and $(L,C,\ev)$ is a strong Michaelis pair, where $\ev$ is the usual evaluation map. In this situation $\coev$ is given by the dual basis.
\end{example}

\begin{example}[Infinite-dimensional Lie algebras]
If $L$ is an infinite-dimensional $k$-Lie algebra, then $L$ is no longer a rigid object in $\Vect(k)$. However, we can still consider the associated Lie monad $-\ot L$, as a YB-Lie algebra in the monoidal category of endofunctors on $\Vect(k)$. As the functor $-\ot L$ has a right adjoint $\Hom_k(L,-)$, the functor $-\ot L$ is right rigid in the category of endofunctors. Hence \prref{charstrongMich} applies and we find that $\Hom_k(L,-)$ is a Lie comonad and $(-\ot L,\Hom_k(L,-))$ is a strong Michaelis pair in the category of endofunctors. Consequently, by \coref{strongMichmonad} the category of representations of the Lie algebra $L$ is equivalent with the category of Lie comodules over the Lie comonad $\Hom_k(L,-)$. This infinite-dimensional example motivates the transition to Lie monads and Lie comonads (hence also YB-Lie algebras, as the category of endofunctors is no longer symmetric).
\end{example}

\section{Dualities between Lie algebras and Hopf algebras}\selabel{DualLieHopf}

\subsection{YB-Lie algebra of primitive elements}\selabel{primitive}
In this section, $\Cc$ is an additive, monoidal category that has equalizers and coequalizers which are preserved by functors of the form $-\ot X$ and $X\ot -$, for any object $X$ in $\Cc$. For the remaining part of this section, we fix a braided bialgebra $H$ in $\Cc$, in the sense of \cite{Tak}. More precisely, we consider a 6-tuple $(H,\mu,\eta,\Delta,\epsilon,\lambda)$ satisfying the following conditions:
\begin{itemize}
\item $(H,\mu,\eta)$ is an algebra in $\Cc$;
\item $(H,\Delta,\epsilon)$ is a coalgebra in $\Cc$;
\item $\lambda$ is an {\em involutive} YB-operator for $H$ (this condition is more restrictive than the usual one of \cite{Tak});
\item The morphism $\lambda$ is compatible with $\mu$ in the sense of \equref{assalgebra}, and in a similar way with $\eta,\Delta$ and $\epsilon$; 
\item $\epsilon: H\to I$ is an algebra morphism; $\eta:I\to H$ is a coalgebra morphism in $\Cc$ and 
\begin{equation}\eqlabel{compatibilitybraidedHopf}
(\mu\ot \mu)\circ (H\ot \lambda\ot H)\circ (\Delta\ot\Delta)=\Delta\circ\mu
\end{equation} 
\end{itemize}

\begin{definition}
The {\em primitive elements} of $H$ are defined as the equalizer $(P(H),\eq)$ in the following diagram
$$\xymatrix{P(H)\ar[rr]^-{\eq} && H \ar@<.5ex>[rr]^-{\Delta}  \ar@<-.5ex>[rr]_-{\eta\ot H+H\ot \eta} &&H\ot H}.$$
\end{definition}

It is well-known that, even in the category of vector spaces, the kernel of the tensor product of two morphisms is not necessarily equal to the tensor product of the kernels of these morphisms, but rather it is a bigger space. Hence the following result might be remarkable at first sight.

\begin{proposition}\prlabel{questionAlessandro}
Let $\Cc$ be a $k$-linear monoidal category as above with $\Char k\neq 3$, and $H$ a braided Hopf algebra in $\Cc$. Put $\alpha=\eta\ot H+H\ot \eta$. 
Then $(P(H)\ot P(H),\eq\ot \eq)$ is the equalizer of $(\Delta\ot\Delta,\alpha\ot\alpha)$.
\end{proposition} 

To prove this theorem, we need the following lemmata.

\begin{lemma}
With notation as above, consider an object $T$ with a morphism $t:T\to H\ot H$ such that $(\Delta\ot \Delta)\circ t=(\alpha\ot \alpha)\circ t$, in other words, $(T,t)$ is an equalizing pair for $(\Delta\ot \Delta,\alpha\ot\alpha)$. Then $(T,t)$ is also an equalizing pair for $(\Delta\ot H,\alpha\ot H)$. 
\end{lemma}

\begin{proof}
First remark that 
\begin{eqnarray}
\alpha\ot \alpha&=& (\eta\ot H+H\ot \eta)\ot (\eta\ot H+H\ot \eta)\nonumber\\
&=&\eta\ot H\ot \eta\ot H + \eta\ot H\ot H\ot \eta + H\ot \eta\ot \eta\ot H + H\ot \eta\ot H\ot \eta\eqlabel{alphaalpha}
\end{eqnarray}
Hence, by the counit property (in the first equality) and combining \equref{alphaalpha} with the fact that $\epsilon\circ \eta=k$ and that $(\epsilon\ot H)\circ(H\ot \eta)=\eta\circ \epsilon=(H\ot \epsilon)\ot(\eta\ot H)$ (in the last equality), we find
\begin{eqnarray*}
t&=&(\epsilon\ot H\ot \epsilon\ot H)\circ (\Delta\ot \Delta)\circ t\\
&=&(\epsilon\ot H\ot \epsilon\ot H)\circ (\alpha\ot \alpha)\circ t\\
&=&(H\ot H+H\ot \eta\circ\epsilon + \eta\circ\epsilon\ot H+\eta\circ\epsilon\ot\eta\circ\epsilon)\circ t
\end{eqnarray*}
If we now apply $(\epsilon\ot\epsilon)$ to this obtained equality, we find
\begin{eqnarray*}
(\epsilon\ot \epsilon)\circ t&=& 4(\epsilon\ot\epsilon)\circ t
\end{eqnarray*}
Hence, since $\Char k\neq 3$, $(\epsilon\ot \epsilon)\circ t=0$. We use this in the following computation, where we apply $(H\ot \epsilon)$, again to the equality above.
\begin{eqnarray*}
(H\ot \epsilon)\circ t&=& (2H\ot \epsilon + 2\eta\circ\epsilon\ot \epsilon)\circ t\\
&=& 2(H\ot \epsilon)\circ t
\end{eqnarray*}
We can conclude that 
\begin{equation*}(H\ot \epsilon)\circ t=0.\eqlabel{epst}\end{equation*}
Finally, we can show that $(T,t)$ is an equalizing pair, as stated.
\begin{eqnarray*}
(\Delta\ot H)\circ t&=&(H\ot H\ot \epsilon\ot H)\circ (\Delta\ot\Delta)\circ t\\
&=& (H\ot H\ot \epsilon\ot H)\circ (\alpha\ot\alpha)\circ t\\
&=& (\eta\ot H\ot H + \eta\ot H\ot \eta\circ \epsilon + H\ot \eta\ot H+ H\ot \eta\ot \eta\circ\epsilon)\circ t\\
&=& ((\eta\ot H + H\ot \eta)\ot H)\circ t + ((\eta\ot H +H\ot \eta)\ot \eta\circ \epsilon)\circ t\\
&=& (\alpha\ot H)\circ t +(\alpha\ot \eta)\circ (H\ot \epsilon)\circ t
=(\alpha\ot H)\circ t
\end{eqnarray*}
\end{proof}

In a symmetric way, one shows that $(T,t)$ is also an equalizing pair for $(\Delta\ot H,\alpha\ot H)$. Let us put $P(H)=P$ and recall that equalizers in $\Cc$ are preserved by tensoring with any object. Hence, $(P\ot H,\eq\ot H)$ and $(H\ot P,H\ot \eq)$ are equalizer of the pairs $(\Delta\ot H,\alpha\ot H)$ and $(H\ot \Delta,H\ot \alpha)$ respectively.  
Therefore, we find unique morphisms $e_1:T\to P\ot H$ and $e_2:T\to H\ot P$ such that $t=(\eq\ot H)\circ e_1=(H\ot \eq)\circ e_2$.

\begin{lemma}\lelabel{QA2} $(T,e_1)$ is an equalizing pair for $(P\ot \Delta,P\ot \alpha)$ (respectively, $(T,e_2)$ is an eqlizing pair for $(\Delta\ot P,\alpha\ot P)$). 
\end{lemma}
\begin{proof}
We compute
\begin{eqnarray*}
(\eq\ot H\ot H)\circ (P\ot \Delta)\circ e_1&=& (H\ot \Delta)\circ (\eq\ot H)\circ e_1 \\
&=& (H\ot \Delta)\circ t \\
&=& (H\ot \alpha)\circ t \\
&=& (H\ot \alpha)\circ (\eq\ot H)\circ e_1 \\
&=& (\eq\ot H\ot H)\circ (P\ot \alpha)\circ e_1
\end{eqnarray*}
Since $(P\ot H\ot H,\eq\ot H\ot H)$ is an equalizer, $\eq\ot H\ot H$ is a monomorphism and $(P\ot \Delta)\circ e_1=(P\ot \alpha)\circ e_1$ as needed.
\end{proof} 

We can now easily prove what was announced earlier.

\begin{proof}[Proof of \prref{questionAlessandro}]
Again, by the fact that equalizers in $\Cc$ are preserved by tensoring with objects and by \leref{QA2}, we now find unique morphisms $e,e':T\to P\ot P$ such that $(P\circ \eq)\circ e=e_1$ and $(\eq\circ P)\circ e'=e_2$, respectively. We now claim that $e=e'=:u$ and this is the unique map with the property that $(\eq\circ \eq)\circ u=t$. Clearly, both $e$ and $e'$ satisfy this property. To show the uniqueness, suppose that $v:T\to P\ot P$ is any morphism such that $(\eq\ot \eq)\circ v=t$. We will show that $v=e$. To show this, it suffices to show that $(P\circ \eq)\circ v=e_1$. 
We find 
\begin{eqnarray*}
(\eq\ot H)\circ (P\ot \eq)\circ v &=& (\eq\ot \eq)\circ v = t\\
&=& (\eq\ot H)\circ e_1 
\end{eqnarray*}
Since $(P\ot H,\eq\ot H)$ is an equalizer, hence $\eq\ot H$ is a monomorphism, the claim follows.

So indeed, $(P\ot P,\eq\ot \eq)$ is the equalizer of $(\Delta\ot \Delta,\alpha\ot \alpha)$.
\end{proof}

Our next aim is to show that the primitive elements are endowed with the structure of a YB-Lie algebra. First, let us search for an involutive YB-operator $\lambda_{P(H)}$ for $P(H)$. Such a morphism $\lambda_{P(H)}: P(H)\ot P(H)\to P(H)\ot P(H)$ will be constructed out of the commutativity of the following diagrams:

$$\xymatrix{
H\ot H\ar[rr]^-{\Delta\ot \Delta}\ar[d]_{\lambda_{H}}&&H\ot H\ot H\ot H\ar[d]_{\lambda_{H\ot H}}\\
H\ot H\ar[rr]^-{\Delta\ot\Delta}&&H\ot H\ot H\ot H}\qquad
\xymatrix{
H\ot H\ar[rr]^-{\alpha}\ar[d]_{\lambda_{H}}&&H\ot H\ot H\ot H\ar[d]_{\lambda_{H\ot H}}\\
H\ot H\ar[rr]^-{\alpha}&&H\ot H\ot H\ot H}$$

where $\lambda_{H\ot H}=(H\ot \lambda_{H}\ot H)\circ(\lambda_{H}\ot\lambda_{H})\circ(H\ot \lambda_{H}\ot H)$ and $\alpha=(\eta\ot H+H\ot \eta)\ot(\eta\ot H+H\ot \eta)$. Indeed, since $(P(H)\ot P(H), \eq\ot\eq)$ is again an equalizer in $\Cc$, the commutativity of the above diagrams implies the existence of a unique morphism $\lambda_{P(H)}: P(H)\ot P(H)\to P(H)\ot P(H)$ such that
\begin{equation}\eqlabel{P(H)eq}
(\eq\ot \eq)\circ \lambda_{P(H)}=\lambda_{H}\circ(\eq\ot \eq),
\end{equation}
by the universal property in the definition of equalizer. So, let us check that these two diagrams commute.
We start with the diagram on the left:
\begin{eqnarray*}
&&(H\ot \lambda_{H}\ot H)\circ(\lambda_{H}\ot\lambda_{H})\circ(H\ot \lambda_{H}\ot H)\circ \Delta\ot \Delta\\
&=&(\Delta\ot H)\ot(H\ot \lambda_{H})\ot(H\ot \Delta\ot \Delta)\circ(\lambda_{H}\ot\lambda_{H})\circ(H\ot \lambda_{H}\ot H)\\
&=&(\Delta\ot H)\ot(H\ot \lambda_{H})\circ(\lambda_{H}\ot H)\ot (\Delta\ot H\ot H)\\
&=&\lambda_{H}\circ (H\ot \Delta)\ot (\Delta\ot H\ot H)
\end{eqnarray*}
All of these three equalities use the fact that $\Delta$ is compatible with $\lambda_{H}$.
\\We now consider the diagram on the right hand side. Let us check that $(\eta\ot H\ot\ \eta\ot H)\circ\lambda_{H}=\lambda_{H\ot H}\circ (\eta\ot H\ot\ \eta\ot H)$:
\begin{eqnarray*}
&&(H\ot \lambda_{H}\ot H)\circ(\lambda_{H}\ot\lambda_{H})\circ(H\ot \lambda_{H}\ot H)\circ(\eta\ot H\ot\eta\ot H)\\
&=&(H\ot \lambda_{H}\ot H)\circ(\lambda_{H}\ot\lambda_{H})\circ(\eta\ot\eta\ot H\ot H)\\
&=&(H\ot \lambda_{H}\ot H)\circ(\eta\ot \eta\ot H\ot H)\circ\lambda_{H}\\
&=&(\eta\ot H\ot\ \eta\ot H)\circ\lambda_{H}
\end{eqnarray*}
All of these equalities use the compatibility of $\eta$ with $\lambda_{H}$. Similar computations are made for the other three components of $\alpha$.
\begin{lemma}\lelabel{YPoperatorP(H)}
$\lambda_{P(H)}$ is an involutive YB-operator for $P(H)$.
\end{lemma}
\begin{proof}
We have, by the universal property of the equalizer, that 
$$(\eq\ot\eq)\circ \lambda_{P(H)}\circ\lambda_{P(H)}= \lambda_{H}\circ \lambda_{H}\circ (\eq\ot\eq)=\eq\ot\eq.$$
In the second equality we use that $\lambda_{H}$ is involutive.
Since $\eq\ot\eq$ is a monomorphism in $\Cc$, it follows that $\lambda_{P(H)}\circ \lambda_{P(H)}=P(H)\ot P(H)$. 
\\We also have that
\begin{eqnarray*}
&&(\eq\ot\eq\ot\eq)\circ(\lambda_{P(H)}\ot P(H))\circ(P(H)\ot \lambda_{P(H)})\circ(\lambda_{P(H)}\ot P(H))\\
&=&(\lambda_{H}\ot H)\circ(H\ot \lambda_{H})\circ(\lambda_{H}\ot H)\circ(\eq\ot\eq\ot\eq) \\
&=&(H\ot \lambda_{H})\circ(\lambda_{H}\ot H)\circ(H\ot\lambda_{H})\circ(\eq\ot\eq\ot\eq) \\
&=&(\eq\ot\eq\ot\eq)\circ(P(H)\ot \lambda_{P(H)})\circ(\lambda_{P(H)}\ot P(H))\circ(P(H)\ot\lambda_{P(H)})
\end{eqnarray*}
In the first and third equality we use the universal property, whereas in the second one, we use the fact that $\lambda(H)$ is a YB-operator for $H$.
Since $\eq\ot\eq\ot\eq$ is also a monomorphism, we find that $\lambda_{P(H)}$ is a YB-operator as well.
\end{proof}

A braided Hopf algebra $H$ in $\Cc$ is in particular a YB-algebra in $\Cc$. Hence, applying the functor $\Ll:\YBAlg(\Cc)\to \YBLieAlg(\Cc)$ of \conref{commutator} it follows that $\lie_{H}=\mu\circ({H\ot H}-\lambda_{H})$ determines a YB-Lie algebra structure on $H$. 
\\We now wish to construct a Lie-bracket $\lie_{P(H)}$ for $P(H)$, inherited from the bracket $\lie_{H}$ we have for $H$.
This is done very similarly to the construction of $\lambda_{P(H)}$, as described above; by universal property-arguments and using the compatibilty conditions of $H$, together with \equref{compatibilitybraidedHopf}, one verifies the existence of a unique morphism $\lie_{P(H)}:P(H)\ot P(H)\to P(H)$ such that
\begin{equation}\eqlabel{P(H)bracket}
\eq\circ \lie_{P(H)}=\lie_{H}\circ(\eq\ot \eq).
\end{equation}
Moreover, using the fact that $\lie_{H}$ is a Lie-bracket for $H$ and keeping in mind that $\eq\ot\eq$ and $\eq\ot\eq\ot\eq$ are both monomorphisms, one shows, in a similar fashion as before, that the conditions \equref{AS}, \equref{Jac} and \equref{compatibility} are satisfied for $\lie_{P(H)}$.

Analogeously, we can consider $(Q(H),\coeq)$,  the ``indecomposables" of $H$, to be the coequalizer \equref{indeccoeq} in $\Cc$.
Summarizing, we have the following
\begin{proposition}
Let $H$ be a braided bialgebra in $\Cc$, then 
\begin{enumerate}[(i)]
\item $(P(H), \lambda_{P(H)}, \lie_{P(H)})$ is a YB-Lie algebra in $\Cc$ and $\eq:P(H)\to \Ll(H)$ is a YB-Lie algebra morphism;
\item $(Q(H),\gamma_{Q(H)},\Gamma_{Q(H)})$ is a YB-Lie coalgebra in $\Cc$ and $\coeq:Q(H)\to \Ll^c(H)$ is a YB-Lie coalgebra morphism. 
\end{enumerate}
Furthermore, these constructions yield functors
$$P:\Bialg\to \YBLieAlg,\qquad Q:\Bialg\to \YBLieCoAlg$$
\end{proposition}
\begin{proof}
The first statement clearly follows from the discussion above. To see that the second statement holds, we need the existence of a YB-operator $\gamma_{Q(H)}$ for $Q(H)$
such that 
\begin{equation}\eqlabel{Q(H)YB}
\gamma_{Q(H)}\circ(\coeq\ot\coeq)=(\coeq\ot\coeq)\circ\gamma_{Q(H)}
\end{equation}
and a morphism $\Gamma_{Q(H)}:Q(H)\to Q(H)\ot Q(H)$ such that
\begin{equation}\eqlabel{Q(H)cobracket}
\coeq\ot\coeq \circ \Gamma_{H}=\Gamma_{Q(H)}\circ\coeq,
\end{equation}
where $\Gamma_{H}=(H\ot H-\lambda_{H,H})\circ \Delta_{H}$, the co-bracket for $H$. For these ingredients to exist and to satisfy the conditions of \deref{LieCoalgebra}, it is sufficient to perform the construction of primitive elements $P(-)$ in the opposite category $\Cc^{op}$ and remark that bialgebras are ``selfdual" objects in a monoidal category, hence bialgebras in $\Cc^{op}$.
\end{proof}
\begin{remark}
When $\Cc$ is the category of $k$-vector spaces over a field $k$, the coequalizer $(Q(H),\coeq)$ coincides with Michaelis' original definition of $Q(H)$, as we remarked in \exref{Liecoalgebra}(\ref{indecomposables}). 
\end{remark}

\subsection{Takeuchi pairs}

Let $(H,\mu_H,\eta_H,\Delta_H,\epsilon_H,\lambda_H)$ and $(K,\mu_{K},\eta_{K},\Delta_{K},\epsilon_{K},\lambda_{K})$ be two braided bialgebras in $\Cc$. We adapt the definition of ``dual pair of bialgebras" (cf. \cite{K} e.g.) to the actual setting, embodied by the following definition:
\begin{definition}\delabel{Takpair}
\begin{enumerate}[(1)]
\item $(H,K,\dual)$ is called a {\em Takeuchi pair} in $\Cc$ if there exists a morphism in $\Cc$ $\dual:H\ot K\to I$, such that the following conditions hold:
\begin{enumerate}
\item $\dual\circ(H\ot\eta_{K})=\epsilon_{H}$;
\item $\dual\circ(\eta_{H}\ot K)=\epsilon_{K}$;
\item $\dual\circ(\mu_{H}\ot K)=\dual\circ(H\ot\dual\ot {K})
\circ(H\ot H\ot \Delta_{K})$;
\item $\dual\circ(H\ot \mu_{K})=\dual\circ(H\ot\dual\ot {K})
\circ(\Delta_{H}\ot K\ot K)$; 
\item
$\dual\circ (H\ot\dual\ot K)\circ (\lambda_{H}\ot K\ot K)
=\dual\circ (H\ot\dual\ot K)\circ (H\ot H\ot \lambda_{K}).
$
\end{enumerate}
\item 
A morphism of Takeuchi pairs is a pair $(\phi,\psi):(H,K,\dual)\to (H',K',\dual')$, where $\phi:H\to H'$ and $\psi:K\to K'$ are morphisms of braided bialgebras such that $\dual'=\dual\circ (\phi\ot \psi)$.
\item Takeuchi pairs and their morphisms constitute a category that we denote by $\Tak(\Cc)$.
\end{enumerate} 
\end{definition}

\begin{lemma}\lelabel{nogeencompatibiliteit}
Let $(H,K,\dual)$ be a Takeuchi pair in $\Cc$, then we have the following equality:
$$\dual\circ(
\lie_{H}
\ot K)=\dual\circ(H\ot\dual\ot K)\circ(H\ot H\ot \Gamma_{K}),$$
where $\Gamma_{K}=(K\ot K-\lambda_{K})\circ \Delta_{K}$ and $\lie_H=\mu_H\circ(H\ot H-\lambda_H)$. 
\end{lemma}
\begin{proof}
We compute
\begin{eqnarray*}
&&\dual\circ(
\lie_{H}
\ot K)\\
&=&\dual\circ\big(\mu_H\circ(H\ot H-\lambda_{H})\ot K\big))=\dual\circ(\mu_H\ot K)\circ\big((H\ot H-\lambda_H)\ot K\big)\\
&=&\dual\circ(H\ot\dual\ot {K})\circ(H\ot H\ot \Delta_{K})\circ\big((H\ot H-\lambda_H)\ot K\big)\\
&=&\dual\circ(H\ot\dual\ot {K})\circ\big((H\ot H-\lambda_H)\ot K\ot K\big)\circ(H\ot H\ot \Delta_{K})\\
&=&\dual\circ(H\ot\dual\ot {K})\circ\big(H\ot H\ot (K\ot K-\lambda_K)\big)\circ(H\ot H\ot \Delta_{K})\\
&=&\dual\circ(H\ot\dual\ot K)\circ(H\ot H\ot \Gamma_{K})
\end{eqnarray*}
We used the third condition of \deref{Takpair} in the third equality and the fifth condition of \deref{Takpair} in the third one.
\end{proof}
\begin{proposition}
Let $(H,K,\dual)$ be a Takeuchi pair in $\Cc$, then $(P(H),Q(K),\ev)$ is a Michaelis pair. Moreover, we obtain a functor 
$$\Pp:\Tak(\Cc)\to \Mich(\Cc),\ \Pp(H,K,\dual)=(P(H),Q(K),\ev)$$
\end{proposition}
\begin{proof}
In order to make our notation not too heavy, let us put $P=P(H)$ and $Q=Q(K)$ in what follows.
Let us first look for a suitable morphism $\ev:P\ot Q\to I$. We know that $(Q,\coeq_K)$ is a coequalizer, and as coequalizers are preserved by tensoring in $\Cc$, $(P\ot Q,P\ot \coeq_K)$ is a coequalizer as well. 
$$\xymatrix{P\ot K\ot K \ar@<.5ex>[rr]^-{P\ot \mu_{K}}  \ar@<-.5ex>[rr]_-{P\ot(K\ot\epsilon_{K}+\epsilon_{K}\ot K)} &&P\ot K\ar[rrd]_-{\dual\circ(\eq_{H}\ot K)} \ar[rr]^-{P\ot \coeq_K} && P\ot Q \ar@{.>}[d]^{\ev} \\ &&&&I}.$$
Therefore, if $\dual\circ(\eq_{H}\ot K):P\ot K\to I$ coequalizes the pair $(P\ot \mu_K,P\ot(K\ot\epsilon_{K}+\epsilon_{K}\ot K))$, then the universal property induces a (unique) morphism
$\ev:P\ot Q\to I$ such that 
\begin{equation}\eqlabel{nodighieronder}
\ev\circ(P\ot \coeq)=\dual\circ(\eq\ot K)
\end{equation}
We calculate:
\begin{eqnarray*}
&& \dual\circ(\eq\ot K)\circ(P\ot \mu_{K})=\dual\circ(P\ot \mu)\circ (\eq\ot K\ot K)\\
&=&\dual\circ (H\ot \dual\ot K)\circ (\Delta_{H}\ot K\ot K)\circ (\eq\ot K\ot K)\\
&=&\dual\circ (H\ot \dual\ot K)\circ\big((\eta_{H}\ot H+H\ot\eta_{H})\ot K\ot K\big)\circ (\eq\ot K\ot K)\\
&=&\dual\circ \big(H\ot (K\ot \epsilon_{K}+\epsilon_{K}\ot K)\big)\circ (\eq\ot K\ot K)\\
&=&\dual\circ(\eq\ot K)\circ \big(P\ot (\epsilon_{K}\ot K+K\ot\epsilon_{K})\big),
\end{eqnarray*}
where we use the fourth condition of \deref{Takpair} in the second equality, the definition of the equalizer $(P,\eq)$ in the third equality, 
and the second condition of \deref{Takpair} in the fourth equality. 
\\We now have to prove that the two diagrams, occuring in \deref{duality}, commute. 
Let us start with the proof of the equality
\begin{equation}\eqlabel{tweedediagramma}
\ev\circ(P\ot \ev\ot Q) \circ (P\ot P\ot \gamma_{Q})=
\ev\circ(P\ot\ev\ot Q)\circ (\lambda_{P}\ot Q\ot Q)
\end{equation}
Applying \equref{Q(H)YB} in the first equality, \equref{nodighieronder} in the second and sixth one, the fifth condition of \deref{Takpair} in the fourth one and \equref{P(H)eq} in the fifth equality, we find:
\begin{eqnarray*}
&&\ev \circ(P\ot \ev\ot Q) \circ(P\ot P\ot \gamma_{Q})\circ(P\ot P\ot\coeq\ot \coeq)\\
&=&\ev\ \circ(P\ot \ev\ot Q) \circ(P\ot P\ot\coeq\ot\coeq)\circ(P\ot P\ot \lambda_{K})\\
&=&\dual\circ (H\ot\dual\ot K) \circ (\eq\ot\eq\ot K\ot K)\circ(P\ot P\ot \lambda_{K})\\
&=&\dual\circ (H\ot\dual\ot K) \circ (H\ot H\ot \lambda_{K})\circ (\eq\ot\eq\ot K\ot K)\\
&=&\dual\circ (H\ot\dual\ot K) \circ (\lambda_{H}\ot K\ot K)\circ (\eq\ot\eq\ot K\ot K)\\
&=&\dual\circ (H\ot\dual\ot K) \circ (\eq\ot\eq\ot K\ot K)\circ (\lambda_{P}\ot K\ot K)\\
&=&\ev \circ(P\ot \ev\ot Q) \circ (P\ot P\ot \coeq\ot\coeq) \circ(\lambda_{P}\ot K\ot K)\\
&=&\ev \circ(P\ot \ev\ot Q)\circ(\lambda_{P}\ot K\ot K)\circ(P\ot P\ot \coeq\ot\coeq)
\end{eqnarray*}
As $P\ot P\ot\coeq\ot \coeq$ is an epimorphism in $\Cc$, \equref{tweedediagramma} holds.
\\We now proceed with proving the commutativity of the other diagram. Using 
 \equref{nodighieronder} in the second equality and the sixth one, \equref{P(H)bracket} in the third equality, \leref{nogeencompatibiliteit} in the fourth one, and finally \equref{Q(H)cobracket} in the last equality, we calculate consequently:
\begin{eqnarray*}
&&\ev\circ(\lie_{P}\ot Q) \circ(P\ot P\ot\coeq)
=\ev\circ (P\ot \coeq)\circ(\lie_{P}\ot K)\\
&=&\dual\circ (\eq\ot K)\circ(\lie_{P}\ot K)
=\dual\circ(\lie_{H}\ot K)\circ(\eq\ot\eq\ot K)\\
&=&\dual\circ(H\ot\dual\ot K)\circ(H\ot H\ot \Gamma_{K})\circ(\eq\ot\eq\ot K)\\
&=&\dual\circ(H\ot\dual\ot K)\circ(\eq\ot\eq\ot K\ot K)\circ (P\ot P\ot \Gamma_{K})\\
&=&\ev\circ(P\ot\ev\ot Q)\circ(P\ot P\ot\coeq\ot\coeq)\circ (P\ot P\ot \Gamma_{K})\\
&=&\ev\circ(P\ot\ev\ot Q)\circ(P\ot P\ot \Gamma_{Q})\circ (P\ot P\ot \coeq)
\end{eqnarray*}
As $P\ot P\ot \coeq$ is an epimorphism in $\Cc$, the above is equivalent with the equality we were looking for. This establishes the result.
\end{proof}

\begin{example}
Let $H$ be a Hopf $k$-algebra over a field $k$, and $H^\circ$ its Sweedler dual. Denote by $H'$ the opposite-co-opposite Hopf $k$-algebra of $H^\circ$. 
Then $(H',H,\dual)$ is a Takeuchi pair, where $\dual$ is the usual evaluation map. Hence, we find that $(P(H'),Q(H),\ev)$ is a Michaelis pair, where $\ev$ is again the usual evaluation map. Michaelis \cite{Mich2} proved moreover that $P(H^\circ)\cong Q(H)^*$, i.e. this Michaelis pair is always strong. In \cite{GV2}, we generalize this result in a setting of additive symmetric monoidal categories, so that it applies in particular to Turaev's Hopf group coalgebras.
\end{example}

Given a braided Hopf algebra, recall from \seref{adjointmodules} that there exists an induction functor $\Ind:\Mod(H)\to\LieMod(\Ll(H))$. On the other hand, the YB-Lie algebra morphism $\eq:P(H)\to H$ induces a functor $\LieMod(\Ll(H))\to \LieMod(P(H))$. Therefore, we obtain a combined functor
$$\Mod(H)\to \LieMod(P(H))$$
Dually, for another braided Hopf algebra $K$, we find a functor $\CoMod(K)\to \LieCoMod(Q(K))$. Therefore, given a Takeuchi pair $(H,K,\dual)$ we obtain the following diagram of functors between categories of left (Lie) (co) modules.
\begin{equation}\eqlabel{MichTakModule}
\xymatrix{
\CoMod(K) \ar[rr]^F  \ar[d]^{G'}
&&\Mod(H)  \ar[d]^{G}\\ 
\LieCoMod(Q(K)) \ar[rr]^{F'}  
&& \LieMod(P(H)) 
}
\end{equation}
\begin{theorem}\thlabel{TakMichModules}
Let $(L,K,\dual)$ be a Takeuchi pair, then the diagram of functors \equref{MichTakModule} commutes.
\end{theorem}
\begin{proof}
Consider a $K$-comodule $(M,\rho^{M,K})$. Then we have $F(M,\rho^{M,K})=(M,\rho_{M,H})$, where
$$\xymatrix{\rho_{M,H}:H\ot M\ar[rr]^-{H\ot \rho^{M,K}} && H\ot K\ot M\ar[rr]^-{\dual\ot M} && M}$$
Next, we find $GF(M,\rho^{M,K})=G(M,\rho_{M,H})=(M,\liea_{M,P(H)})$, with
$$\xymatrix{\liea_{M,P(H)}: P(M)\ot M \ar[rr]^-{\eq\ot M} && H\ot M \ar[rr]^-{\rho_{M,H}} && M}$$
On the other hand, we obtain $G'(M,\rho^{M,K})$, where
$$\xymatrix{\liec^{M,Q(K)}:M \ar[rr]^-{\rho^{M,K}} && K\ot M \ar[rr]^-{\coeq \ot M} && Q(K)\ot M}$$
We continue and compute $F'G'(M,\rho^{M,K})=F'(M,\liec^{M,Q(K)})=(M,\liea'_{M,P(H)})$ given by
$$\xymatrix{\liea'_{M,P(H)}: P(H)\ot M \ar[rr]^-{P(H)\ot \liec^{M,Q(K)}} && P(H)\ot Q(K)\ot M \ar[rr]^-{\ev\ot M} && M }$$
Finally, to see that $F'G'=GF$ it suffices to verify that $\liea'_{M,P(H)}=\liea_{M,P(H)}$, i.e.
\begin{eqnarray*}
\liea'_{M,P(H)}&=& (\ev\ot M)\circ (P(H)\ot \coeq\ot M)\circ (P(H)\ot \rho^{M,K})\\
&=&(\dual\ot M)\circ (\eq\ot K\ot M)\circ (P(H)\ot \rho^{M,K})\\
&=&(\dual\ot M)\circ (H\ot \rho^{M,K})\circ (\eq\ot M) = 
\liea_{M,P(H)} 
\end{eqnarray*}
This finishes the proof.
\end{proof}

\subsection*{Acknowledgement}
We would like to thank Alessandro Ardizzoni for useful comments on an earlier version of this paper.

\end{document}